\documentclass[12pt]{article}
\usepackage[margin=3cm]{geometry}
\usepackage[all]{xy}
\usepackage{amsmath, amsfonts, amssymb, amsthm}
\usepackage{hyperref}
\usepackage{amsmath,amsfonts,amssymb,amsthm,epsfig,epstopdf,titling,url,array}
\usepackage{tikz}
\theoremstyle{plain}
\newtheorem{thm}{Theorem}[subsection]
\newtheorem{lem}[thm]{Lemma}
\newtheorem{prop}[thm]{Proposition}
\newtheorem{cor}{Corollary}[subsection]
\theoremstyle{definition}
\newtheorem{defn}{Definition}[subsection]

\newtheorem{exmp}{Example}[subsection]
\theoremstyle{remark}
\newtheorem{remark}{Remark}[subsection]
\newtheorem{note}{Note}[subsection]
\theoremstyle{proof}

\author{SELCAN AKSOY\\Giresun University, Giresun TURKEY \\ E-mail: selcan.aksoy@giresun.edu.tr}
\title{2 Category of FRBSU Monoidal Categories and Crossed Modules }
\begin{document}
\maketitle
\begin{abstract}
In that paper, we prove that the collection of all FRBSU monoidal categories and the collection of all crossed modules form a 2 category.
\end{abstract}
\tableofcontents

\section{Preliminaries}
Throughout this paper, we are assuming that the symbol FRBSU monoidal category indicates a finite, rigid, braided monoidal category whose unit object $I$ is simple.\\

A category is small if the collection of all objects form a set. If $A$ is an object in a category $\mathcal{A}$, then a subobject $B$ of $A$ is an object with a monomorphism $B\rightarrow A$.\\ 

\cite{baki} An object $A$ is simple in an abelian category $\mathcal{A}$ if for any injection $B\rightarrow A$, we get $B=0$ or $B\cong A$.\\

A cover for an object $A$ in a category $\mathcal{A}$ is an object $P$ with an epimorphism $f:~P\rightarrow A$. This cover is projective if $P$ is a projective object.\\

An object $A$ in a category is of finite length if there exists a finite sequence of monomorphisms $\xymatrix{ 0\ar[r] & A_n\ar[r] & A_{n-1}\ar[r] & ...\ar[r] & A_0=A}$ such that the cokernels of these monomorphisms are simple objects.\\

A $k$ linear abelian category is semisimple if every object is isomorphic to direct sum of simple objects.
\begin{lem}
\label{22}
$[Schur's~Lemma]$ If $k$ is an algebraically closed field of characteristic zero, then $End(X)=k$ whenever $X$ is a simple object in an abelian $k$ linear category $\mathcal{A}$.
\end{lem}
\begin{lem}
\label{12}
If $X\cong Y$ are nonzero simple objects in a $k$ linear abelian category $\mathcal{A}$ for $k$ is a perfect field, then $Hom(X,~Y)=0$.
\end{lem}
\begin{proof}
Assume that $\mathcal{A}$ is a $k$ linear abelian category and $X$, $Y$ are nonzero simple objects. Let $f:~X\rightarrow Y$ is a nonzero morphism in $\mathcal{A}$. $Ker(f)\cong 0$ since $X$ is simple and $f\neq 0$, so that morphism is a monomorphism. As a result, $X\cong Y$ that is a contradiction, hence $f=0$.
\end{proof}

A $k$ linear abelian category $\mathcal{A}$ where $k$ is a perfect field is finite if for all objects $X$, $Y$ in $\mathcal{A}$, $Hom_{\mathcal{A}}(X,~Y)$ is finite dimensional vector space over $k$, all objects $A\in \mathcal{A}$ has finite length, every simple object in $\mathcal{A}$ has a projective cover and that category has finitely many isomorphism classes of simple objects.\\

For example $Vec_f(k)$ is a finite category, because $Hom_{Vec_f(k)}(V,~W)$ is isomorphic to the vector space $M_{m\times n}(k)$ of $m\times n$ matrices in which the entries are elements of the field $k$ where $dim(V)=m$ and $dim(W)=n$ for given two finite dimensional vector spaces $V$ and $W$. It is finite dimensional since $M_{m\times n}(k)$ is finite dimensional with dimension $m\times n$. The only simple object is $k$ and every object is free in that category, as a result every object is projective and $k^2\rightarrow k$ is a surjection for example.\\

\subsection{Monoidal Category and Braiding} 
We use \cite{joro} as a reference for the following definitions and example.
\begin{defn}
$(\mathcal{A},~\otimes,~I,~a,~l,~r)$ is a monoidal category if for all objects $X$, $Y$, $Z$ and $W$ in $\mathcal{A}$, the associativity pentagon and the unit triangle commute.\\

Here $\mathcal{A}$ is a category, $\otimes:~\mathcal{A} \times \mathcal{A} \rightarrow \mathcal{A}$ is a functor, $I$ is a unit object in $\mathcal{A}$, $a$ is the associativity constraint which is a family of natural isomorphisms 
\begin{align}
\begin{tikzpicture}
\node (A) at (0, 0) {$a_{XYZ}:~(X\otimes Y) \otimes Z$};
\node (B) at (4, 0) {$X\otimes (Y\otimes Z)$};
\path[->] (A) edge node [above] {$\cong$} (B); 
\end{tikzpicture}
\end{align}

$l$ is a left unit constraint which is a family of natural isomorphisms 
\begin{align}
\begin{tikzpicture}
\node (A) at (0, 0) {$l_X:~I\otimes X$};
\node (B) at (3, 0) {$X$};
\path[->] (A) edge node [above] {$\cong$} (B); 
\end{tikzpicture}
\end{align}

and $r$ is a right unit constraint which is a family of natural isomorphisms 
\begin{align}
\begin{tikzpicture}
\node (A) at (0, 0) {$r_X:~X\otimes I$};
\node (B) at (3, 0) {$X.$};
\path[->] (A) edge node [above] {$\cong$} (B); 
\end{tikzpicture}
\end{align}
\end{defn}
\begin{lem}
If $(\mathcal{A},~\otimes,~I,~a,~l,~r)$ is a monoidal category, then $\mathcal{A}^{op}$ is a monoidal category.
\end{lem} 
\begin{proof}
We define the tensor product as $X\otimes^{op} Y=Y\otimes X$ and associativity constraint $a^{op}$ as a family of natural isomorphisms 
\begin{align*}
\begin{tikzpicture}
\node (A) at (0, 0) {$a^{op}_{XYZ}:~(X\otimes^{op} Y)\otimes^{op} Z$};
\node (B) at (5.5, 0) {$X\otimes^{op}(Y\otimes^{op} Z)$};
\path[->] (A) edge node [above] {$\cong$} (B); 
\end{tikzpicture}
\end{align*}

 in $\mathcal{A}^{op}$ for all objects $X$, $Y$ and $Z$ in $\mathcal{A}$. This is same as the family of natural isomorphisms
\begin{align*}
\begin{tikzpicture}
\node (A) at (0, 0) {$Z\otimes (Y\otimes X)$};
\node (B) at (4, 0) {$(Z\otimes Y) \otimes X$};
\path[->] (A) edge node [above] {$\cong$} (B); 
\end{tikzpicture}
\end{align*}

in $\mathcal{A}^{op}$ which can be obtained by inverting the arrows 
\begin{align*}
\begin{tikzpicture}
\node (A) at (0, 0) {$(Z\otimes Y) \otimes X$};
\node (B) at (4, 0) {$Z\otimes (Y\otimes X)$};
\path[->] (A) edge node [above] {$\cong$} (B); 
\end{tikzpicture}
\end{align*}

in $\mathcal{A}$ and as a result, we get $a^{op}_{XYZ}=a^{-1}_{ZYX}$ for all objects $X$, $Y$ and $Z$ in $\mathcal{A}$.\\

Here, $I^{op}=I$. $l^{op}$ is a left unit constraint which is a family of natural isomorphisms
\begin{align*}
\begin{tikzpicture}
\node (A) at (0, 0) {$l^{op}_X:~I\otimes^{op} X$};
\node (B) at (3, 0) {$X$};
\path[->] (A) edge node [above] {$\cong$} (B); 
\end{tikzpicture}
\end{align*}

in $\mathcal{A}^{op}$ which is same as
\begin{align*}
\begin{tikzpicture}
\node (A) at (0, 0) {$ l^{op}_X:~X\otimes I$};
\node (B) at (3, 0) {$X$};
\path[->] (A) edge node [above] {$\cong$} (B); 
\end{tikzpicture}
\end{align*} 

for all objects $X$ in $\mathcal{A}$. So, we take $l^{op}_X=r_X$ and $l^{op}=r$ in $\mathcal{A}$. Similarly, we can take $r^{op}=l$.
\end{proof}

Also, we define a category $\mathcal{A}^{rev}$ for a given monoidal category $\mathcal{A}$ in which the objects and the arrows are the same as in $\mathcal{A}$ and the tensor product is defined as $X\otimes^{rev} Y=Y\otimes X$.\\

A strictly full subcategory $\mathcal{B}$ of a monoidal category $\mathcal{A}$ is monoidal if it contains the unit object $I$ in $\mathcal{A}$ and $A\otimes B$ for all objects $A$ and $B$ in $\mathcal{A}$.\\

$(\mathcal{A},~\otimes)$ is an additive monoidal category if $\mathcal{A}$ is an additive category and $\otimes$ is a biadditive functor. It is abelian if $\mathcal{A}$ is an abelian category.\\

A monoidal category is strict if all $a$, $l$ and $r$ are identity arrows. For example, the category of all $k$ vector spaces $Vec(k)$ is not a strict monoidal category for a given field $k$. $U\otimes (V\otimes W)\neq (U\otimes V)\otimes W$ in general for all vector spaces $U,~V,~W$ in $Vec(k)$, even in $Vec_f(k)$, but we can obtain a family of natural isomorphisms of those products as an associativity constraint.
\begin{thm}
$[MacLane]$ Every monoidal category is equivalent to a strict monoidal category.
\end{thm}
\begin{defn}
An object $A$ is invertible in a category $\mathcal{A}$ if there exists an object $B$ in $\mathcal{A}$ such that $A\otimes B\cong B\otimes A\cong  I$ for $I$ is the unit object. 
\end{defn}
\begin{remark}
Invertible objects in a monoidal category $\mathcal{A}$ form a monoidal subcategory of that category. If every simple object in $\mathcal{A}$ is invertible, then we say that the category is pointed. 
\end{remark}
\begin{defn}
A braiding $c$ for a monoidal category $\mathcal{A}$ is a natural family of isomorphisms $\xymatrix{c_{XY}:~X\otimes Y \ar[rr]^{\cong} & & Y\otimes X}$ for all objects $X$, $Y$ in $\mathcal{A}$ such that two hexagon diagrams commute in $\mathcal{A}$.
\end{defn}
\begin{note}
If $\mathcal{A}$ is a braided monoidal category, then $\mathcal{A}^{rev}$ is a braided monoidal category with the braiding $c^{rev}$ that is a family of natural isomorphisms $c^{rev}_{XY}=c_{YX}$ for all objects $X$ and $Y$ in $\mathcal{A}$. Similarly, $\mathcal{A}^{op}$ is a braided monoidal category with the braiding $c^{op}$ that is a family of natural isomorphisms $c^{op}_{XY}=c^{-1}_{XY}$ for all objects $X$ and $Y$ in $\mathcal{A}$. $\mathcal{A}^{op}\simeq \mathcal{A}^{rev}$ in that situation.
\end{note}
\begin{defn}
A monoidal category $\mathcal{A}$ with a braiding $c$ is called symmetric if the composition 
\begin{align}
\xymatrix{ X\otimes Y \ar[rr]^{c_{XY}}  & & Y\otimes X \ar[rr]^{c_{YX}} & & X\otimes Y}
\end{align}

is $id_{X\otimes Y}$ for all objects $X$, $Y$ in $\mathcal{A}$.
\end{defn}
\begin{exmp}
The category of all $k$ vector spaces $Vec(k)$ is a braided, symmetric monoidal category for a field $k$.
\end{exmp}
\begin{proof}
$(c_{YX}\circ c_{XY})(x\otimes y)=(c_{YX}\circ c_{XY})(xy)=c_{YX}(yx)=xy=x\otimes y$ for all objects $X$ and $Y$ in $Vec(k)$, for all elements $x\in X,~y\in Y$. As a result, the composition is the identity.
\end{proof}
\begin{exmp}
Assume that $G$ is an abelian group, $k$ is a field, $(f,~h)$ is an abelian 3 cocycle on $G$ with coefficients in $k$. That is, $f:~G\times G\times G\rightarrow k$ is a normalized 3 cocycle such that 
\begin{align}
f(x,~0,~y)=0,
\end{align}
\begin{align}
f(x,~y,~z)+f(w,~x+y,~z)+f(w,~x,~y)=f(w,~x,~y+z)+f(w+x,~y,~z)
\end{align}

and $h:~G\times G\rightarrow k$ is a function such that 
\begin{align}
f(y,~z,~x)+h(x,~y+z)+f(x,~y,~z)=h(x,~z)+f(y,~x,~z)+h(x,~y),
\end{align}
\begin{align}
-f(z,~x,~y)+h(x+y,~z)-f(x,~y,~z)=h(x,~z)-f(x,~z,~y)+h(y,~z).
\end{align}

Let $\mathcal{A}$ be a category such that the objects are families of $k$ modules $X=\{X_g~|~g\in G\}$ and the arrow between two families $X$, $Y$ is a family $\xymatrix{ \theta=\{X_{g_1}\ar[r]^{\theta_{g_1g_2}} & Y_{g_2}\}}$ where $\theta_{g_1g_2}$ is a $k$ module homomorphism for all $g_1$ and $g_2$ in $G$,
\begin{align}
(X\otimes Y)_g=\underset{g_1+g_2=g}{\Sigma}(X_{g_1}\otimes Y_{g_2})
\end{align}

is the tensor product,
\begin{align}
a_{XYZ}:~(X\otimes Y)\otimes Z\rightarrow X\otimes (Y\otimes Z),\\ a_{XYZ}((x\otimes y)\otimes z)=f(g_1,~g_2,~g_3)x\otimes (y\otimes z)
\end{align}

is the associativity constraint and 
\begin{align}
c_{XY}:~X\otimes Y\rightarrow Y\otimes X,\\ c(x\otimes y)=h(g_1,~g_2)y\otimes x
\end{align}

is the braiding for $x\in X_{g_1}$, $y\in Y_{g_2}$, $z\in Z_{g_3}$. So, this category is a braided monoidal category.
\end{exmp} 

\subsection{The Category of Monoidal Functors}
The following materials are found in \cite{joro}.
\begin{defn}
For two monoidal categories $\mathcal{A}$ and $\mathcal{B}$, assume that $\xymatrix{ \mathcal{F}:~\mathcal{A} \ar[r] & \mathcal{B}}$ is a functor, $\gamma$ is the family of natural isomorphisms 
\begin{align}
\begin{tikzpicture}
\node (A) at (0, 0) {$\gamma_{XY}:~\mathcal{F}(X) \otimes \mathcal{F}(Y)$};
\node (B) at (4, 0) {$\mathcal{F}(X\otimes Y)$};
\path[->] (A) edge node [above] {$\cong$} (B); 
\end{tikzpicture}
\end{align} 

for all objects $X$, $Y$ in $\mathcal{A}$, $\xymatrix{ \varphi:~I\ar[r]^{\cong} & \mathcal{F}(I)}$ is an isomorphism for the unit object $I$. Then, $(\mathcal{F},~\gamma,~\varphi)$ is a monoidal functor if it satisfies compatible conditions.
\end{defn}
\begin{note}
$(\mathcal{F},~\gamma,~\varphi)$ is strict if $\gamma$ and $\varphi$ are identities.
\end{note}
\begin{defn}
A monoidal functor $\mathcal{F}:~\mathcal{A} \rightarrow \mathcal{B}$ between braided monoidal categories $\mathcal{A}$ and $\mathcal{B}$ is braided if the following diagram is commutative.
\begin{equation}
\xymatrix{ \mathcal{F}(X) \otimes \mathcal{F}(Y) \ar[d]^c \ar[r]^{\gamma} & \mathcal{F}(X\otimes Y) \ar[d]^{\mathcal{F}(c)}\\ \mathcal{F}(Y) \otimes \mathcal{F}(X) \ar[r]^{\gamma} & \mathcal{F}(Y\otimes X) }
\end{equation}
\end{defn}
\begin{defn}
If $\mathcal{F},~\mathcal{G}:~\mathcal{A} \rightarrow \mathcal{B}$ are two monoidal functors, then a map $\theta:~\mathcal{F} \rightarrow \mathcal{G}$ is a natural transformation if the following two diagrams commute.
\begin{equation}
\xymatrix@C+1em { \mathcal{F}(X) \otimes \mathcal{F}(Y) \ar[d]^{\theta(X) \otimes \theta(Y)} \ar[r]^{\gamma} & \mathcal{F}(X\otimes Y) \ar[d]^{\theta(X\otimes Y)}\\ \mathcal{G}(X) \otimes \mathcal{G}(Y) \ar[r]^{\gamma} & \mathcal{G}(X\otimes Y) } \quad \xymatrix{ & I \ar[ld]_{\varphi} \ar[rd]^{\varphi}\\ \mathcal{F}(I) \ar[rr]^{\theta(I)} & & \mathcal{G}(I) }
\end{equation}
\end{defn}
\begin{prop}
The collection $Hom(\mathcal{A},~\mathcal{B})$ in which the objects are monoidal functors $\mathcal{F}:~\mathcal{A} \rightarrow \mathcal{B}$ and morphisms are natural transformations between monoidal functors for given monoidal categories $\mathcal{A}$ and $\mathcal{B}$ forms a category.
\end{prop}
\begin{lem}
$Hom(\mathcal{A},~\mathcal{A})$ is a monoidal category in which the tensor product is the composition of functors for a given monoidal category $\mathcal{A}$.
\end{lem}

We denote the category of right exact monoidal functors by $Hom^{re}(\mathcal{A},~\mathcal{B})$, the category of left exact monoidal functors by $Hom^{le}(\mathcal{A},~\mathcal{B})$ and the category of exact monoidal functors by $Hom^{e}(\mathcal{A},~\mathcal{B})$.
\begin{remark}
A monoidal functor $(\mathcal{F},~\gamma,~\varphi):~(\mathcal{A},~\otimes_{\mathcal{A}}) \rightarrow (\mathcal{B},~\otimes_{\mathcal{B}})$ is a monoidal equivalence if $\mathcal{F}:~\mathcal{A}\rightarrow \mathcal{B}$ is an equivalence of categories. In that situation, there exists a monoidal functor $(\mathcal{G},~\gamma',~\varphi'):~\mathcal{B} \rightarrow \mathcal{A}$ and isomorphism of monoidal functors $\mathcal{G}\circ \mathcal{F} \rightarrow id_{\mathcal{A}}$, $\mathcal{F}\circ \mathcal{G} \rightarrow id_{\mathcal{B}}$.
\end{remark}
\begin{prop}
\label{32}
\cite{joro} If $\mathcal{A}$ is braided monoidal category, then we get a monoidal equivalence $\mathcal{A} \rightarrow \mathcal{A}^{rev}$.
\end{prop}
\begin{proof}
We define a monoidal functor $\mathcal{F}:~\mathcal{A} \rightarrow \mathcal{A}^{rev}$ by sending an object $A$ in $\mathcal{A}$ to itself, $\gamma$ as a family of natural isomorphisms $\gamma_{XY}:~X\otimes^{rev} Y=Y\otimes X \rightarrow X\otimes Y$ for all objects $X$, $Y$ in $\mathcal{A}$ and also $\varphi=id_I$. We define $\gamma_{XY}=c^{rev}_{XY}$. Then, we need to show that the following diagram commutes in $\mathcal{A}^{rev}$.
\begin{align}
\label{354}
\xymatrix{ (X\otimes^{rev} Y)\otimes^{rev} Z \ar[d]_{\gamma \otimes^{rev} id} \ar[r]^{a_{XYZ}^{rev}} & X\otimes^{rev} (Y\otimes^{rev} Z) \ar[d]^{id\otimes^{rev} \gamma} \\ (X\otimes Y)\otimes^{rev} Z\ar[d]_{\gamma} & X\otimes^{rev} (Y\otimes Z) \ar[d]^{\gamma} & &\\ (X\otimes Y)\otimes Z\ar[r]_{a_{XYZ}} & X\otimes (Y\otimes Z)}
\end{align}

This diagram is same as the following diagram in $\mathcal{A}$.
\begin{align}
\label{400}
\xymatrix{ Z\otimes (Y\otimes X) \ar[d]_{id_Z\otimes c_{YX}} \ar[r]^{a^{-1}_{ZYX}} & (Z\otimes Y) \otimes X\ar[d]^{c_{ZY}\otimes id_X} \\ Z\otimes (X\otimes Y) \ar[d]_{c_{Z(XY)}} & (Y\otimes Z) \otimes X \ar[d]^{c_{(YZ)X}} & & \\ (X\otimes Y)\otimes Z \ar[r]_{a_{XYZ}} & X\otimes (Y\otimes Z)}
\end{align}

The first and third squares commute by definition and the middle one commutes by using naturality of the braiding in the following diagram.
\begin{align}
\label{353}
\begin{tikzpicture}
\node (A) at (0, 0) {$(Z\otimes Y)\otimes X$};
\node (B) at (4, 0) {$(Y\otimes Z)\otimes X$};
\node (C) at (8, 0) {$Y\otimes (Z\otimes X)$};
\node (D) at (12, 0) {$Y\otimes (X\otimes Z)$};
\node (E) at (0, -2) {$Z\otimes (Y\otimes X)$};
\node (F) at (12, -2) {$(Y\otimes X)\otimes Z$};
\node (G) at (0, -4) {$Z\otimes (X\otimes Y)$};
\node (H) at (12, -4) {$(X\otimes Y)\otimes Z$};
\node (K) at (0, -6) {$(Z\otimes X)\otimes Y$};
\node (L) at (4, -6) {$(X\otimes Z)\otimes Y)$};
\node (M) at (8, -6) {$X\otimes (Z\otimes Y)$};
\node (N) at (12, -6) {$X\otimes (Y\otimes Z)$};
\path[->] (A) edge node [above] {$c_{ZY}\otimes id_X$} (B);
\path[->] (B) edge node [above] {$a_{YZX}$} (C);
\path[->] (C) edge node [above] {$id_Y\otimes c_{ZX}$} (D);
\path[->] (E) edge node [above] {$c_{Z(YX)}$} (F);
\path[->] (G) edge node [above] {$c_{Z(XY)}$} (H);
\path[->] (K) edge node [below] {$c_{ZX}\otimes id_Y$} (L);
\path[->] (L) edge node [below] {$a_{XZY}$} (M);
\path[->] (M) edge node [below] {$id_X\otimes c_{ZY}$} (N);
\path[->] (A) edge node [left] {$a_{ZYX}$} (E);
\path[->] (E) edge node [left] {$id_Z\otimes c_{YX}$} (G);
\path[->] (G) edge node [left] {$a^{-1}_{ZXY}$} (K);
\path[->] (D) edge node [right] {$a^{-1}_{YXZ}$} (F);
\path[->] (F) edge node [right] {$c_{YX}\otimes id_Z$} (H);
\path[->] (H) edge node [right] {$a_{XYZ}$} (N);
\end{tikzpicture}
\end{align}

As a result, that diagram commutes.\\

$a_{XYZ}\circ c_{Z(XY)}\circ (id_Z\otimes c_{YX})=a_{XYZ}\circ a^{-1}_{XYZ}\circ (id_X\otimes c_{ZY})\circ a_{XZY}\circ (c\otimes id_Y)\circ a^{-1}_{ZXY}\circ (id_Z\otimes c_{YX})=(id_X\otimes c_{ZY})\circ a_{XZY}\circ (c\otimes id_Y)\circ a^{-1}_{ZXY}\circ (id_Z\otimes c_{YX})$.\\

$c_{(YZ)X}\circ (c_{ZY}\otimes id_X)\circ a^{-1}_{ZYX}=a_{XYZ}\circ (c_{YX}\otimes id_Z)\circ a^{-1}_{YXZ}\circ (id_Y\otimes c_{ZX})\circ a_{YZX}\circ (c_{ZY}\otimes id_X)\circ a^{-1}_{ZYX}$.\\

These two equations are same by the commutativity of Diagram \ref{353}. As a result, Diagram \ref{354} commutes. The commutativity of other diagrams are easy to show, also it is a braided monoidal functor. The reader can show that the conditions for equivalence are satisfied.
\end{proof}

$\mathcal{A} \simeq \mathcal{A}^{op}$ as a corollary of this proposition.

\subsection{Rigid Monoidal Categories}
\cite{baki} An object $Y$ in a monoidal category $\mathcal{A}$ is a right dual for a given object $X$ in $\mathcal{A}$ if there are morphisms $ev_{rX}:~Y\otimes X \rightarrow I$ and $coev_{rX}:~I \rightarrow X\otimes Y$ such that the following compositions are the identities.
\begin{align}
\begin{tikzpicture}
\node (A) at (0, 0) {$Y=Y\otimes I$};
\node (B) at (5, 0) {$Y\otimes X \otimes Y$};
\node (C) at (10, 0) {$I\otimes Y=Y$};
\path[->] (A) edge node [above] {$id_Y\otimes coev_{rX}$} (B); 
\path[->] (B) edge node [above] {$ev_{rX}\otimes id_Y$} (C); 
\end{tikzpicture}
\\
\label{97}
\begin{tikzpicture}
\node (A) at (0, 0) {$X=I\otimes X $};
\node (B) at (5, 0) {$X\otimes Y \otimes X$};
\node (C) at (10, 0) {$X\otimes I=X$};
\path[->] (A) edge node [above] {$coev_{rX}\otimes id_X$} (B); 
\path[->] (B) edge node [above] {$id_X\otimes ev_{rX}$} (C); 
\end{tikzpicture}
\end{align}

Similarly, an object $Z$ is a left dual object for the object $X$ in that category if there are morphisms $ev_{lX}:~X\otimes Z\rightarrow I$ and $coev_{lX}:~I\rightarrow Z\otimes X$ such that the following compositions are the identities.
\begin{align}
\label{98}
\begin{tikzpicture}
\node (A) at (0, 0) {$Z=I\otimes Z$};
\node (B) at (5, 0) {$Z\otimes X\otimes Z$};
\node (C) at (10, 0) {$Z\otimes I=Z$};
\path[->] (A) edge node [above] {$coev_{lX}\otimes id_Z$} (B);
\path[->] (B) edge node [above] {$id_Z\otimes ev_{lX}$} (C);
\end{tikzpicture}
\\
\label{99}
\begin{tikzpicture}
\node (A) at (0, 0) {$X=X\otimes I$};
\node (B) at (5, 0) {$X\otimes Z\otimes X$};
\node (C) at (10, 0) {$I\otimes X=X$};
\path[->] (A) edge node [above] {$id_X\otimes coev_{lX}$} (B);
\path[->] (B) edge node [above] {$ev_{lX}\otimes id_X$} (C);
\end{tikzpicture}
\end{align}

We denote the left dual with $^+X$ and the right dual with $X^+$.
\begin{lem}
\label{96}
A left dual $^+X$ and a right dual $X^+$ in a monoidal category $\mathcal{A}$ is unique up to a unique isomorphism.
\end{lem}
\begin{proof}
See \cite{baki} for the proof.
\end{proof}
\begin{defn}
A monoidal category is rigid if every object $X$ in that category has both a right and a left dual object.
\end{defn}
\begin{exmp}
The category of finite dimensional vector spaces $Vec_k$ over a field $k$ is rigid. If that category is consisting of all $k$ vector spaces without finiteness assumption, then it is not rigid.
\end{exmp}
\begin{proof}
For a given finite dimensional $k$ vector space $V$, the right and left dual object for $V$ is the dual space $Hom_k(V,~k)$ with the evaluation map 
\begin{align*}
ev_{rV}:~Hom_k(V,~k)\otimes V\rightarrow k,~(f,~v)\mapsto f(v)
\end{align*}

and the coevaluation map $coev_{rV}:~k\rightarrow V\otimes Hom_k(V,~k)$ which is an embedding. We may see that the following compositions are the identities.
\begin{align*}
\begin{tikzpicture}
\node (A) at (0,0) {$V=k\otimes V$};
\node (B) at (6,0) {$V\otimes Hom_k(V,~k)\otimes V$};
\node (C) at (12,0) {$V\otimes k=V$};
\path[->] (A) edge [left] node [above] {$coev_{rV}\otimes id_V$} (B);
\path[->] (B) edge [right] node [above] {$id_V\otimes ev_{rV}$} (C);
\end{tikzpicture}
\end{align*}
\begin{align*}
\begin{tikzpicture}
\node (A) at (0,0) {$Hom_k(V,~k)\otimes k$};
\node (B) at (6,0) {$Hom_k(V,~k)\otimes V\otimes Hom_k(V,~k)$};
\node (C) at (12,0) {$k\otimes Hom_k(V,~k)$};
\node (D) at (0,-2) {$Hom_k(V,~k)$};
\node (E) at (6, -2) {$Hom_k(V,~k)\otimes V\otimes Hom_k(V,~k)$};
\node (F) at (12,-2) {$Hom_k(V,~k)$};
\draw[thick, double] (0, -0.2)--(0, -1.8) [xshift=5pt];
\draw[thick, double] (12, -0.2)--(12, -1.8) [xshift=5pt];
\path[->] (A) edge [left] node [above] {$id\otimes coev_{rV}$} (B);
\path[->] (B) edge [right] node [above] {$ev_{rV}\otimes id$} (C);
\path[->] (D) edge [left] node [above] {$id\otimes coev_{rV}$} (E);
\path[->] (E) edge [right] node [above] {$ev_{rV}\otimes id$} (F);
\end{tikzpicture}
\end{align*}

Similarly, we may show that the compositions \ref{98} and \ref{99} are the identities which shows that $Hom_k(V,~k)$ is a left dual for the object $V$.\\

Second part follows since infinite dimensional spaces don't have any coevaluation map.  
\end{proof}
\begin{remark}
\label{82}
\cite{baki} Tensor product functor $\otimes:~\mathcal{A} \times \mathcal{A} \rightarrow \mathcal{A}$ is exact in each variable in an abelian, rigid monoidal category $\mathcal{A}$.
\end{remark}
\begin{lem}
If a monoidal category $\mathcal{A}$ is rigid, then $\mathcal{A}^{op}$ is rigid, too.
\end{lem}
\begin{proof}
If $X$ is an object in a rigid monoidal category $\mathcal{A}$, then $X$ has both a left and a right dual objects $^+X$ and $X^+$ which are unique up to a unique isomorphism by Lemma \ref{96} such that the following compositions are the identities by definition where $ev_{rX}:~X_+\otimes X \rightarrow I$, $coev_{rX}:~I \rightarrow X\otimes X^+$, $ev_{lX}:~X\otimes {^+X}\rightarrow I$ and  $coev_{lX}:~I \rightarrow {^+X}\otimes X$  are morphisms in $\mathcal{A}$.
\begin{align*}
\begin{tikzpicture}
\node (A) at (0, 0) {$X^+=X^+\otimes I$};
\node (B) at (6, 0) {$X^+\otimes X\otimes X^+$};
\node (C) at (12, 0) {$I\otimes X^+=X^+$};
\path[->] (A) edge node [above] {$id_{X^+}\otimes coev_{rX}$} (B); 
\path[->] (B) edge node [above] {$ev_{rX}\otimes id_{X^+}$} (C); 
\end{tikzpicture}
\\
\begin{tikzpicture}
\node(A) at (0, 0) {$X=I\otimes X$};
\node (B) at (6, 0) {$X\otimes X^+\otimes X$};
\node (C) at (12, 0) {$X\otimes I=X$};
\path[->] (A) edge node [above] {$coev_{rZ}\otimes id_X$} (B);
\path[->] (B) edge node [above]{$id_X\otimes ev_{rX}$} (C);
\end{tikzpicture}
\\
\begin{tikzpicture}
\node (A) at (0, 0) {$^+X=I\otimes {^+X}$};
\node (B) at (6, 0) {$^+X\otimes X \otimes {^+X}$};
\node (C) at (12, 0) {$ ^+X\otimes I={^+X}$};
\path[->] (A) edge node [above] {$coev_{lX}\otimes id_{^+X}$} (B); 
\path[->] (B) edge node [above] {$id_{^+X}\otimes ev_{lX}$} (C); 
\end{tikzpicture}
\\
\begin{tikzpicture}
\node (A) at (0, 0) {$X=X\otimes I$};
\node (B) at (6, 0) {$X\otimes {^+X} \otimes X$};
\node (C) at (12, 0) {$I\otimes X=X$};
\path[->] (A) edge node [above] {$id_X\otimes coev_{lX}$} (B); 
\path[->] (B) edge node [above] {$ev_{lX}\otimes id_X$} (C); 
\end{tikzpicture}
\end{align*}

We know that $ev_{rX}\otimes id_{X^+} \circ id_{X^+}\otimes coev_{rX}=id_{X^+}$ in $\mathcal{A}$ by the first composition, so $(ev_{rX}\otimes id_{X^+} \circ id_{X^+}\otimes coev_{rX})^{op}=(id_{X^+})^{op}=id_{X^+}$. This implies that 
\begin{align*}
(id_{X^+}\otimes coev_{rX})^{op}\circ (ev_{rX}\otimes id_{X^+})^{op}=coev_{rX}^{op}\otimes ^{op} id_{X^+}\circ id_{X^+}\otimes ^{op} ev_{rX}^{op}=id_{X^+}.
\end{align*}

As a result, the following composition is the identity of $X^+$ in $\mathcal{A}^{op}$
\begin{align*}
\begin{tikzpicture}
\node (A) at (0, 0) {$X^+=I\otimes X^+$};
\node (B) at (6, 0) {$X^+\otimes X \otimes X^+$};
\node (C) at (12, 0) {$X^+\otimes I=X^+$};
\path[->] (A) edge node [above] {$id_{X^+}\otimes^{op} ev_{rX}^{op}$} (B); 
\path[->] (B) edge node [above] {$coev_{rX}^{op}\otimes^{op} id_{X^+}$} (C); 
\end{tikzpicture}
\end{align*}

which is same as the following one
\begin{align*}
\begin{tikzpicture}
\node (A) at (0, 0) {$X^+=X^+\otimes^{op} I $};
\node (B) at (6, 0) {$ X^+\otimes^{op} X \otimes^{op} X^+$};
\node (C) at (12, 0) {$I\otimes^{op} X^+=X^+.$};
\path[->] (A) edge node [above] {$id_{X^+}\otimes^{op} ev_{rX}^{op}$} (B); 
\path[->] (B) edge node [above] {$coev_{rX}^{op}\otimes^{op} id_{X^+}$} (C); 
\end{tikzpicture}
\end{align*}

Also, we get an identity of $X$ by the composition
\begin{align*}
\begin{tikzpicture}
\node (A) at (0, 0) {$X=X\otimes I$};
\node (B) at (6, 0) {$X\otimes X^+\otimes X$};
\node (C) at (12, 0) {$ I\otimes X=X$};
\path[->] (A) edge node [above] {$(id_X\otimes ev_{rX})^{op}$} (B); 
\path[->] (B) edge node [above] {$(coev_{rX}\otimes id_X)^{op}$} (C); 
\end{tikzpicture}
\end{align*}

which is same as the following one
\begin{align*}
\begin{tikzpicture}
\node (A) at (0, 0) {$X=I\otimes^{op} X$};
\node (B) at (6, 0) {$X\otimes^{op} X^+\otimes^{op} X$};
\node (C) at (12, 0) {$X\otimes^{op} I=X.$};
\path[->] (A) edge node [above] {$ev_{rX}^{op}\otimes^{op} id_X$} (B); 
\path[->] (B) edge node [above] {$id_X\otimes^{op} coev_{rX}^{op}$} (C); 
\end{tikzpicture}
\end{align*}

These two identities show that $X^+$ is right dual for $X$ in $\mathcal{A}^{op}$. By using same technique, we may show that $^+X$ is left dual for $X$ in $\mathcal{A}^{op}$. Those are unique objects in $\mathcal{A}$ by Lemma \ref{96}, so they are also unique in $\mathcal{A}^{op}$ as objects. This shows that $\mathcal{A}^{op}$ is a rigid category.
\end{proof}
\begin{lem}
If $I$ is a unit object in a rigid monoidal category $\mathcal{A}$, then $I^+=I$.
\end{lem}
\begin{proof}
The composition $\xymatrix{I\ar[r] & I^+\ar[r] & I}$ is the identity by \ref{97}. Also, the other condition is satisfied. It is easy to see that $I$ satisfies the required conditions, too. Hence, $I=I^+$ by uniqueness of a right dual.
\end{proof}
\begin{lem}
\label{71}
\cite{baki} If a monoidal category $\mathcal{A}$ is rigid, then for all objects $X$ and $Y$ in $\mathcal{A}$, $(X\otimes_{\mathcal{A}} Y)^+=Y^+\otimes_{\mathcal{A}} X^+$.
\end{lem}
\begin{exmp}
Assume that $\mathcal{A}$ and $\mathcal{B}$ are two monoidal categories and $(\mathcal{F},~\gamma,~\varphi)$ is a monoidal functor between those categories. If $X$ is an object in $\mathcal{A}$ with a right dual $X^+$, then $\mathcal{F}(X^+)$ is a right dual of $\mathcal{F}(X)$.
\end{exmp}
\begin{proof}
We define the evaluation map as $ev_{r\mathcal{F}(X)}=\mathcal{F}(ev_{rX}) \circ \gamma$ that is shown with the following diagram
\begin{align*}
ev_{r\mathcal{F}(X)}:~\mathcal{F}(X^+)\otimes \mathcal{F}(X)\rightarrow \mathcal{F}(X^+\otimes X) \rightarrow \mathcal{F}(I)
\end{align*}

and the coevaluation map as $coev_{r\mathcal{F}(X)}=\gamma^{-1} \circ \mathcal{F}(coev_{rX})$ that is shown with the following diagram
\begin{align*}
coev_{r\mathcal{F}(X)}:~\mathcal{F}(I) \rightarrow \mathcal{F}(X\otimes  X^+) \rightarrow \mathcal{F}(X) \otimes \mathcal{F}(X^+)
\end{align*}

by using $\gamma$. It is obvious that the following compositions are the identities since $\mathcal{F}$ is a monoidal functor and $X^+$ is a right dual for $X$.
\begin{align*}
\xymatrix{ \mathcal{F}(X^+)=\mathcal{F}(X^+) \otimes \mathcal{F}(I) \ar[r] & \mathcal{F}(X^+) \otimes \mathcal{F}(X) \otimes \mathcal{F}(X^+) \ar[r] & \mathcal{F}(I) \otimes \mathcal{F}(X^+)=\mathcal{F}(X^+)}\\
\xymatrix{ \mathcal{F}(X)=\mathcal{F}(I) \otimes \mathcal{F}(X) \ar[r] & \mathcal{F}(X) \otimes \mathcal{F}(X^+) \otimes \mathcal{F}(X) \ar[r] & \mathcal{F}(X) \otimes \mathcal{F}(I)=\mathcal{F}(X)}
\end{align*}

As a result, $\mathcal{F}(X^+)$ is a right dual for $\mathcal{F}(X)$.
\end{proof}
\begin{defn}
A monoidal subcategory of a monoidal category $\mathcal{A}$ is a monoidal category under the induced monoidal structure of $\mathcal{A}$ and it is a rigid monoidal subcategory of a rigid monoidal category $\mathcal{A}$ if it contains $X^+$ and $^+X$ whenever it contains an object $X$. 
\end{defn}
\begin{prop}
\label{66}
If $\mathcal{A}$ is a rigid monoidal category, then an object $X$ in $\mathcal{A}$ is invertible if and only if $ev_{rX}:~X^+\otimes X\rightarrow I$ and $coev_{rX}:~I\rightarrow X\otimes X^+$ are isomorphisms. In that situation, $^+X\cong X^+$. If $Y$ is another invertible object, then $X\otimes Y$ is invertible.
\end{prop}
\begin{proof}
If the above maps are isomorphisms, then we get $X^+\otimes X\cong I\cong X\otimes X^+$, so $X^+$ is the required object in the definition of an invertible object. Thus, $X$ is invertible. Similarly, we see that $X^+$ is invertible.\\

Conversely, if $X$ is invertible, then there exists an object $Z$ such that $X\otimes Z\cong Z\otimes X\cong I$, so we can use $Z$ as a right dual, hence $Z\cong X^+$ by uniqueness of a right dual. Then the above maps are isomorphisms. With the same idea, we may consider $Z\cong {^+X}$ by the isomorphism and we reverse the arrows if required and see $Z$ is a left dual. As a result, $X^+\cong {^+X}$.\\

Now, assume that $X$ and $Y$ are two invertible objects in the category. Then, $X^+\otimes X\cong I\cong X\otimes X^+$ and $Y^+\otimes Y\cong I\cong Y\otimes Y^+$. So, we get $Y^+\otimes X^+\otimes X\otimes Y\cong I\cong X\otimes Y\otimes Y^+\otimes X^+$.  This shows that $Y^+\otimes X^+$ is the inverse of $X\otimes Y$.   
\end{proof}
\begin{prop}
Assume that $(\mathcal{F},~\gamma,~\varphi)$ and $(\mathcal{G},~\gamma',~\varphi')$ are two monoidal functors from $\mathcal{A} \rightarrow \mathcal{B}$. If $\mathcal{A}$ and $\mathcal{B}$ are rigid monoidal categories, then every morphism of monoidal functors from $\mathcal{F}$ to $\mathcal{G}$ is an isomorphism. 
\end{prop}

\subsection{ Drinfeld Center of A Monoidal Category}
Assume that $\mathcal{A}$ is a monoidal category. We denote the Drinfeld center of $\mathcal{A}$ by $\mathcal{Z}(\mathcal{A})$.\\

Objects in $\mathcal{Z}(\mathcal{A})$ are $(Z,~\gamma_Z)$ where $Z$ is an object in $\mathcal{A}$ and $\gamma_Z$ is a family of natural isomorphisms $\xymatrix{ \gamma_{ZX}:~Z\otimes X \ar[rr]^{\sim} & & X\otimes Z}$ for all objects $X$ in $\mathcal{A}$ such that the following diagrams commute.
\begin{align}
\begin{tikzpicture}
\node (A) at (1, 0) {$Z\otimes (X\otimes Y)$};
\node (B) at (-2, -1.5) {$(Z\otimes X)\otimes Y$};
\node (C) at (1, -3) {$(X\otimes Z)\otimes Y$};
\node (D) at (5, -3) {$X\otimes (Z\otimes Y)$};
\node (E) at (8, -1.5) {$X\otimes(Y\otimes Z)$};
\node (F) at (5, 0) {$(X\otimes Y)\otimes Z$};
\draw [->, thick] (4, -1.3) arc (360:20:20pt);
\path[->] (A) edge node [above] {$\gamma_{Z(X\otimes Y)}$} (F);
\path[->] (B) edge node [left, midway] {$a$} (A);
\path[->] (B) edge node [left, midway] {$\gamma_{ZX} \otimes id_Y$} (C);
\path[->] (C) edge node [below] {$a$} (D);
\path[->] (D) edge node [right, midway] {$id_X\otimes \gamma_{ZY}$} (E);
\path[->] (F) edge node [right, midway] {$a$} (E);
\end{tikzpicture}
\end{align}
\begin{align}
\begin{tikzpicture}
\node (A) at (0, 0) {$Z\otimes I$};
\node (B) at (4, 0) {$I\otimes Z$};
\node (C) at (2,-2) {$Z$};
\path[->] (A) edge node [above] {$\gamma_{ZI}$} (B);
\path[->] (A) edge node [left]  {$r$} (C);
\path[->] (B) edge node [right] {$l$} (C);
\end{tikzpicture}
\end{align}

A morphism $f:~(Z,~\gamma_Z) \rightarrow (W,~\gamma_W)$ in $\mathcal{Z}(\mathcal{A})$ is an arrow $f:~Z \rightarrow W$ such that the following diagram commutes.
\begin{equation}
\xymatrix{ Z\otimes X\ar[d]_{f\otimes id_X} \ar[rr]^{\gamma_{ZX}} & & X\otimes Z \ar[d]^{id_X\otimes f}\\ W\otimes X \ar[rr]_{\gamma_{WX}} & & X\otimes W}
\end{equation}
\begin{lem}
$Z(\mathcal{A})$ is a braided monoidal category.
\end{lem}
\begin{proof}
$(Z,~\gamma_Z)\otimes (W,~\gamma_W)=(Z\otimes W,~(\gamma_Z \otimes id_W)\circ (id_Z\otimes \gamma_W))$ is a tensor product for all objects $Z$ and $W$ in $\mathcal{A}$ and $c_{(Z,~\gamma_Z)(W,~\gamma_W)}: (Z,~\gamma_Z)\otimes (W,~\gamma_W) \rightarrow (W,~\gamma_W)\otimes (Z,~\gamma_Z)$ is a braiding which is same as 
\begin{align*}
c:~(Z\otimes W,~(\gamma_Z \otimes id_W) \circ (id_Z\otimes \gamma_W)) \rightarrow (W\otimes Z,~(\gamma_W \otimes id_Z)\circ (id_W\otimes \gamma_Z)).
\end{align*}
\end{proof}

\subsection{Module Categories}
\begin{defn}
\cite{os} A category $\mathcal{M}$ is a left module category on a finite monoidal category $\mathcal{A}$ if there exists an exact bifunctor
\begin{align}
\label{74}
\otimes_{l\mathcal{M}}:~\mathcal{A} \times \mathcal{M} \rightarrow \mathcal{M},~(X,~M)\mapsto X\otimes_{l\mathcal{M}} M
\end{align}

for all objects $X$ in $\mathcal{A}$, $M$ in $\mathcal{M}$, associativity constraint $a_{l\mathcal{M}}$ consisting of associativity isomorphisms $a_{XYM}: (X\otimes_{\mathcal{A}} Y)\otimes_{l\mathcal{M}} M\rightarrow X\otimes_{l\mathcal{M}}(Y\otimes_{l\mathcal{M}} M)$ and left unit constraint $l_{\mathcal{M}}$ which is a family of left unit isomorphisms $l_M:~I\otimes_{l\mathcal{M}} M\rightarrow M$ such that for any $X$, $Y$, $Z$ in $\mathcal{A}$, $M$ in $\mathcal{M}$, the following diagrams commute.
\begin{align}
\begin{tikzpicture}
\node (A) at (0, 0) {$((X\otimes_{\mathcal{A}} Y)\otimes_{\mathcal{A}} Z)\otimes_{l\mathcal{M}} M$};
\node (B) at (8, 0) {$(X\otimes_{\mathcal{A}} (Y\otimes_{\mathcal{A}} Z))\otimes_{l\mathcal{M}} M$};
\node (C) at (-1, -2) {$(X\otimes_{\mathcal{A}} Y)\otimes_{l\mathcal{M}} (Z\otimes_{l\mathcal{M}} M) $};
\node (D) at (9, -2) {$X\otimes_{l\mathcal{M}} ((Y\otimes_{\mathcal{A}} Z)\otimes_{l\mathcal{M}} M)$};
\node (E) at (4, -4) {$X\otimes_{l\mathcal{M}} (Y\otimes_{l\mathcal{M}} (Z\otimes_{l\mathcal{M}} M))$};
\draw [->, thick] (4.8, -1.5) arc (360:30:20pt);
\path[->] (A) edge node [above] {$a_{XYZ}\otimes_{l\mathcal{M}} id_M$} (B);
\path[->] (A) edge node [left, midway] {$a_{(XY)ZM}$} (C);
\path[->] (B) edge node [right, midway] {$a_{X(YZ)M}$} (D);
\path[->] (C) edge node [left, midway] {$a_{XY(ZM)}$} (E);
\path[->] (D) edge node [right, midway] {$id_X\otimes_{l\mathcal{M}} a_{YZM}$} (E);
\end{tikzpicture}
\end{align}
\begin{align}
\begin{tikzpicture}
\node (A) at (0, 0) {$(X\otimes_{\mathcal{A}} I)\otimes_{l\mathcal{M}} M$};
\node (B) at (6, 0) {$ X\otimes_{l\mathcal{M}} (I\otimes_{l\mathcal{M}} M)$};
\node (C) at (3, -2) {$X\otimes_{l\mathcal{M}} M$};
\path[->] (A) edge node [above] {$a_{XIM}$} (B);
\path[->] (A) edge node [left] {$r_X\otimes_{l\mathcal{M}} id_M$} (C);
\path[->] (B) edge node [right] {$id_X\otimes_{l\mathcal{M}} l_M$} (C);
\end{tikzpicture}
\end{align}
\end{defn}
\begin{defn}
A category $\mathcal{M}$ is right module category on a finite monoidal category $\mathcal{A}$ if there exists an exact bifunctor
\begin{align*}
\otimes_{r\mathcal{M}}:~\mathcal{M} \times \mathcal{A} \rightarrow \mathcal{M},~(M,~X)\mapsto M\otimes_{r\mathcal{M}} X
\end{align*}

for all objects $X\in \mathcal{A}$, $M\in \mathcal{M}$, associativity constraint $a_{r\mathcal{M}}$ consisting of associativity isomorphisms $a_{MXY}:~M\otimes_{r\mathcal{M}} (X\otimes_{\mathcal{A}} Y) \rightarrow (M\otimes_{r\mathcal{M}} X)\otimes_{r\mathcal{M}} Y$ and right unit constraint $r_{\mathcal{M}}$ which is a family of right unit isomorphisms $r_M:~M\otimes_{r\mathcal{M}} I\rightarrow M$ such that for any $X$, $Y$, $Z$ in $\mathcal{A}$, $M$ in $\mathcal{M}$, the following diagrams commute.
\begin{align}
\label{69}
\begin{tikzpicture}
\node (A) at (0, 0) {$M\otimes_{r\mathcal{M}} ((X\otimes_{\mathcal{A}} Y)\otimes_{\mathcal{A}} Z)$};
\node (B) at (8, 0) {$M\otimes_{r\mathcal{M}} (X\otimes_{\mathcal{A}} (Y\otimes_{\mathcal{A}} Z))$};
\node (C) at (-1, -2) {$(M\otimes_{r\mathcal{M}} X) \otimes_{r\mathcal{M}}(Y\otimes_{\mathcal{A}} Z)$};
\node (D) at (9, -2) {$(M\otimes_{r\mathcal{M}} (X\otimes_{\mathcal{A}} Y))\otimes_{r\mathcal{M}} Z$};
\node (E) at (4, -4) {$((M\otimes_{r\mathcal{M}} X) \otimes_{r\mathcal{M}} Y)\otimes_{r\mathcal{M}} Z$};
\draw [->, thick] (4.8, -1.5) arc (360:30:20pt);
\path[->] (A) edge node [above] {$_{id_M\otimes_{r\mathcal{M}} a_{XYZ}}$} (B);
\path[->] (A) edge node [left, midway] {$a_{M(XY)Z}$} (C);
\path[->] (B) edge node [right, midway] {$a_{MX(YZ)}$} (D);
\path[->] (C) edge node [left, midway] {$a_{(MX)YZ}$} (E);
\path[->] (D) edge node [right, midway] {$a_{MXY}\otimes_{r\mathcal{M}} id_Z$} (E);
\end{tikzpicture}
\end{align}
\begin{align}
\label{70}
\begin{tikzpicture}
\node (A) at (0, 0) {$M\otimes_{r\mathcal{M}} (I\otimes_{\mathcal{A}} X)$};
\node (B) at (6, 0) {$(M\otimes_{r\mathcal{M}} I)\otimes_{r\mathcal{M}} X$};
\node (C) at (3, -2) {$M\otimes_{r\mathcal{M}} X$};
\path[->] (A) edge node [above] {$a_{MIX}$} (B);
\path[->] (A) edge node [left] {$id_M\otimes_{r\mathcal{M}} l_X$} (C);
\path[->] (B) edge node [right] {$r_M\otimes_{r\mathcal{M}} id_X$} (C);
\end{tikzpicture}
\end{align}
\end{defn}
\begin{prop} 
\label{1}
For a left $\mathcal{A}$ module category $\mathcal{M}$ where $\mathcal{A}$ is a finite, rigid monoidal category, $\mathcal{M}^{op}$ is a right $\mathcal{A}$ module category obtained from $\mathcal{M}$ by reversing the arrows.
\end{prop}
\begin{proof}
We define the action as 
\begin{align*}
\otimes_{r\mathcal{M}^{op}}:~\mathcal{M}^{op} \times \mathcal{A} \rightarrow \mathcal{M}^{op},~(M,~X)\mapsto M\otimes_{r\mathcal{M}^{op}} X
\end{align*}

for all objects $M\in \mathcal{M}^{op}$ and for all objects $X\in \mathcal{A}$ where $M\otimes_{r\mathcal{M}^{op}} X=X^+\otimes_{l\mathcal{M}} M$. $\otimes_{r\mathcal{M}^{op}}$ is an exact bifunctor since $\otimes_{l\mathcal{M}}$ is an exact bifunctor.\\

Can we find an associativity constraint $a_{r\mathcal{M}^{op}}$ consisting of associativity isomorphisms $a_{MXY}:~M\otimes_{r\mathcal{M}^{op}} (X\otimes_{\mathcal{A}} Y) \rightarrow (M\otimes_{r\mathcal{M}^{op}} X)\otimes_{r\mathcal{M}^{op}} Y$ for all objects $X$, $Y$ in $\mathcal{A}$, $M$ in $\mathcal{M}$ and a right unit constraint $r_{\mathcal{M}^{op}}$ which is a family of right unit isomorphisms $r_M:~M\otimes_{r\mathcal{M}^{op}} I\rightarrow M$ such that for any $X$, $Y$, $Z$ in $\mathcal{A}$, $M$ in $\mathcal{M}$, the Diagram \ref{69} and Diagram \ref{70} commute?
\begin{align*}
M\otimes_{r\mathcal{M}^{op}}(X\otimes_{\mathcal{A}} Y)=(X\otimes_{\mathcal{A}} Y)^+\otimes_{l\mathcal{M}} \mathcal{M}=(Y^+\otimes_{\mathcal{A}} X^+)\otimes_{l\mathcal{M}} M~by~Lemma~\ref{71},
\end{align*}
\begin{align*}
(M\otimes_{r\mathcal{M}^{op}} X)\otimes_{r\mathcal{M}^{op}} Y=Y^+\otimes_{l\mathcal{M}}(X^+\otimes_{l\mathcal{M}} M).
\end{align*}

We know that there is an isomorphism 
\begin{align*}
a_{Y^+X^+M}:~(Y^+\otimes_{\mathcal{A}} X^+)\otimes_{l\mathcal{M}} M\rightarrow Y^+\otimes_{l\mathcal{M}}(X^+\otimes_{l\mathcal{M}} M)
\end{align*}

since $\mathcal{M}$ is a left $\mathcal{A}$ module category, so, we can take $a_{MXY}=a_{Y^+X^+M}$.\\

$M\otimes_{r\mathcal{M}^{op}} I=I^+\otimes_{l\mathcal{M}} M$ and there is an isomorphism $l_M:~I^+\otimes_{l\mathcal{M}} M\rightarrow M$, so we get the family of right unit constraints $r_M=l_M$ since $I^+=I$. The reader can show that the Diagram \ref{69} and Diagram \ref{70} commute.
\end{proof}

Similarly, for a right $\mathcal{A}$ module category $\mathcal{M}$, $\mathcal{M}^{op}$ is the category obtained from $\mathcal{M}$ by reversing the arrows which is a left $\mathcal{A}$ module category with $X\times M=M\otimes {^+X}$ for all objects $M\in \mathcal{M}$, $X\in \mathcal{A}$.
\begin{lem}
Assume that $\mathcal{A}$ is a fusion category and $\mathcal{M}$ is a module category over $\mathcal{A}$. $(\mathcal{M}^{op})^{op}\simeq \mathcal{M}$ canonically as an $\mathcal{A}$ module category.
\end{lem}
\begin{lem}
\label{43}
Assume that $\mathcal{A}$ is a finite monoidal category and $A$ is an algebra in $\mathcal{A}$. Then, the category of right $A$ modules $\mathcal{A}A$ is a left $A$ module category and the category of left $A$ modules $A\mathcal{A}$ is a right $A$ module category.
\end{lem}
\begin{note}
A tensor category $\mathcal{A}$ is a $Z(\mathcal{A})$ module category with the action 
\begin{align*}
Z(\mathcal{A})\times \mathcal{A} \rightarrow \mathcal{A},~((Z,~\gamma),~X)\mapsto Z\otimes X.
\end{align*}
\end{note}

\subsection{Indecomposable, Exact and Semisimple Module Category}
\begin{prop}
Assume that $\mathcal{M}$ and $\mathcal{N}$ are module categories over a finite monoidal category $\mathcal{A}$. Then, the direct sum $\mathcal{P}=\mathcal{M} \oplus \mathcal{N}$ of $\mathcal{M}$ and $\mathcal{N}$ is a module category over $\mathcal{A}$ with $\otimes_{\mathcal{P}}=\otimes_{\mathcal{M}} \oplus \otimes_{\mathcal{N}}, \quad a_{\mathcal{P}}=a_{\mathcal{M}} \oplus a_{\mathcal{N}}, \quad l_{\mathcal{P}}=l_{\mathcal{M}} \oplus l_{\mathcal{N}}$.
\end{prop}
\begin{defn}
A module category $\mathcal{P}$ over a finite monoidal category $\mathcal{A}$ is indecomposable if $\mathcal{M}=0$ or $\mathcal{N}=0$ whenever $\mathcal{P} \simeq \mathcal{M} \oplus \mathcal{N}$.
\end{defn}
\begin{defn}
A module category $\mathcal{M}$ over a monoidal category $\mathcal{A}$ is exact if for all projective objects $P$ in $\mathcal{A}$ and all objects $M$ in $\mathcal{M}$, $P\otimes M$ is a projective object in $\mathcal{M}$.
\end{defn}
\begin{lem}
\label{72}
Every finite monoidal category $\mathcal{A}$ is an exact module category over itself.
\end{lem}
\begin{exmp}
Every object in an exact module category $\mathcal{M}$ over $Vec_f(k)$ is projective.
\end{exmp}
\begin{proof}
$k$ is free, hence a projective object in $Vec_f(k)$, for all objects $M$ in $\mathcal{M}$, we have $M=k\otimes M$ that is projective. As a result, every object is projective in $\mathcal{M}$.
\end{proof}
\begin{lem}
If $\mathcal{M}$ is a module category over a rigid monoidal category $\mathcal{A}$, then for all objects $A\in \mathcal{A}$ and projective objects $P\in \mathcal{M}$, $A\otimes P$ is a projective object in $\mathcal{M}$.
\end{lem}
\begin{proof}
Assume that $\mathcal{M}$ is a module category over $\mathcal{A}$ given as above. Let $A$ be an object in $\mathcal{A}$ and $P$ be projective in $\mathcal{M}$. For all epimorhisms $f:~X\rightarrow Y$ and morphisms $g:~A\otimes P\rightarrow Y$ in $\mathcal{M}$, can we find a morphism $k:~A\otimes P\rightarrow X$ such that $g=k\circ f$.
\begin{align*}
Hom_{\mathcal{M}}(X,~Y)\cong Hom_{\mathcal{M}}(^+A\otimes X,~^+A\otimes Y),\\
Hom_{\mathcal{M}}(A\otimes P,~Y)\cong Hom_{\mathcal{M}}(P,~^+A\otimes Y),\\
Hom_{\mathcal{M}}(A\otimes P,~X)\cong Hom_{\mathcal{M}}(P,~^+A\otimes X).
\end{align*}

We find an epimorphism $f':~{^+A}\otimes X\rightarrow {^+A}\otimes Y$ corresponding to $f$ and a morphism $g':~P\rightarrow {^+A}\otimes Y$ corresponding to $g$. Since $P$ is projective, then we get a morphism $k':~P\rightarrow {^+A}\otimes X$ such that $g'=f'\circ k'$. After that, we get a morphism $k:~A\otimes P\rightarrow X$ corresponding to $k'$ such that $f\circ k=g$. As a result, $A\otimes P$ is a projective object in $\mathcal{M}$.
\end{proof}
\begin{lem}
\label{77}
If $\mathcal{A}$ is a finite semisimple monoidal category, then the unit object $I$ is projective in that category. Plus, all objects are projective in such a category.
\end{lem}
\begin{proof}
We want to show that $I$ is projective in $\mathcal{A}$. Assume that we are given an epimorphism $f:~A\rightarrow B$ and a map $g:~I\rightarrow B$. Can we find a map $h:~I\rightarrow A$ such that $f\circ h=g$?\\

$A=\oplus_i A_i$, $I=\oplus_j I_j$ and $B=\oplus_k B_k$ for simple objects $A_i$, $I_j$ and $B_k$ for all $i$, $j$, and $k$. We can decompose $f$ and $g$ as $f=~\oplus_{ik} f_{ik}$ and $g=~\oplus_{jk} g_{jk}$ where $f_{ik}:~A_i\rightarrow B_k$, $g_{jk}:~I_j\rightarrow B_k$ are morphisms in $\mathcal{A}$. By Proposition \ref{12}, we get $f_{ik}=g_{jk}=0$, so any morphism $I_j\rightarrow A_i$ works, then we take $h$ as the direct sum of those morphisms and the result follows afterwards.\\

$A=I\otimes A$ for all objects $A$ in $\mathcal{A}$ by left unit associativity. $\mathcal{A}$ is an exact left module category over itself by Lemma \ref{72}. $I$ is projective, thus $I\otimes A$ is projective by exactness. Hence, every object is projective in $\mathcal{A}$.
\end{proof}
\begin{cor}
If a module category $\mathcal{M}$ over a finite monoidal category $\mathcal{A}$ is semisimple, then it is exact.
\end{cor}
\begin{proof}
Assume that $\mathcal{M}$ is a semisimple module category over a finite monoidal category $\mathcal{A}$. Any object in a semisimple category is projective, so $\mathcal{M}$ is exact.
\end{proof}
\begin{cor}
\cite{etos} A module category $\mathcal{M}$ over a fusion category $\mathcal{A}$ is exact if and only if it is semisimple.
\end{cor}
\begin{lem}
Assume that $\mathcal{A}$ is a finite, rigid monoidal category with simple unit object $I$. Then, any exact module category $\mathcal{M}$ over $\mathcal{A}$ has enough projectives.
\end{lem}
\begin{proof}
Assume that $\mathcal{A}$ is given as above and $\mathcal{M}$ is an exact left module category over $\mathcal{A}$. Then, we find a projective object $P$ in $\mathcal{A}$ with an epimorphism $P\rightarrow I$ since $\mathcal{A}$ is a finite category. Hence, for all objects $M$ in $\mathcal{M}$, we get an epimorphism $P\otimes M\rightarrow I\otimes M\cong M$ in $\mathcal{M}$. $\mathcal{M}$ is exact, so $P\otimes M$ is projective by exactness. As a result, we see that $\mathcal{M}$ has enough projectives and there exists a projective cover for every simple object in $\mathcal{M}$.   
\end{proof}
 
\subsection{The Category of Module Functors}
\begin{defn}
\cite{os} Assume that $\mathcal{M}$ and $\mathcal{N}$ are two left module categories over a finite monoidal category $\mathcal{A}$. A module functor between them is a pair $(\mathcal{F},~f)$ where $\mathcal{F}:~\mathcal{M} \rightarrow \mathcal{N}$ is a functor and $f$ is a family of natural isomorphisms 
\begin{align}
\label{13} 
f_{XM}:~\mathcal{F}(X\otimes M)\rightarrow X\otimes \mathcal{F}(M)
\end{align}

for all objects $X$ in $\mathcal{A}$ and $M$ in $\mathcal{M}$ such that for any $X$, $Y$ in $\mathcal{A}$, $M$ in $\mathcal{M}$, the following diagrams are commutative.
\begin{equation}
\xymatrix{ \mathcal{F}((X\otimes Y)\otimes M) \ar[d]^{f_{(X\otimes Y)M}} \ar[rr]^{\mathcal{F}(a_{XYZ})} & & \mathcal{F}(X\otimes (Y\otimes M)) \ar[rr]^{f_{X(Y\otimes M)}} & & X\otimes \mathcal{F}(Y\otimes M) \ar[d]^{id_X\otimes f_{YM}}\\ (X\otimes Y)\otimes \mathcal{F}(M) \ar[rrrr]_{a_{XY\mathcal{F}(M)}} & & & & X\otimes(Y\otimes \mathcal{F}(M)) }
\end{equation}
\begin{equation}
\xymatrix{ \mathcal{F}(I\otimes M) \ar[rd]_{\mathcal{F}(l_{M})} \ar[rr]^{f_{IM}} & & I\otimes \mathcal{F}(M) \ar[ld]^{l_{\mathcal{F}(M)}} \\ & \mathcal{F}(M) }
\end{equation}
\end{defn}

The collection of all module functors $(\mathcal{F},~f):~\mathcal{M} \rightarrow \mathcal{N}$ between two module categories $\mathcal{M}$ and $\mathcal{N}$ over a finite monoidal category $\mathcal{A}$ is denoted by $Hom_{\mathcal{A}}(\mathcal{M},~\mathcal{N})$.
\begin{lem}
If $(\mathcal{F},~f):~\mathcal{M} \rightarrow \mathcal{N}$ and $(\mathcal{G},~g):~\mathcal{N} \rightarrow \mathcal{K}$ are two module functors, then $(\mathcal{G} \circ \mathcal{F},~e):~\mathcal{M} \rightarrow \mathcal{K}$ is a module functor where $e=g\circ \mathcal{G}(f)$.
\end{lem}

 A morphism between $(\mathcal{F},~f)$ and $(\mathcal{G},~g)$ is a natural transformation $h:~\mathcal{F} \rightarrow \mathcal{G}$ such that for any $X$ in $\mathcal{A}$, $M$ in $\mathcal{M}$, the following diagram commutes.
\begin{equation}
\xymatrix{ \mathcal{F}(X\otimes M) \ar[d]_{f_{XM}} \ar[rr]^{h(X\otimes M)} & & \mathcal{G}(X\otimes M) \ar[d]^{g_{XM}}\\ X\otimes \mathcal{F}(M) \ar[rr]^{id_X\otimes h(M)} & & X\otimes \mathcal{G}(M) }
\end{equation}
\begin{lem}
$Hom_{\mathcal{A}}(\mathcal{M},~\mathcal{N})$ is a category of module functors $(\mathcal{F},~f):~\mathcal{M} \rightarrow \mathcal{N}$ for all module categories $\mathcal{M}$ and $\mathcal{N}$ over a given finite monoidal category $\mathcal{A}$. 
\end{lem}

Two module categories $\mathcal{M}$ and $\mathcal{N}$ over a finite monoidal category $\mathcal{A}$ are equivalent if there exist module functors $(\mathcal{F},~f):~\mathcal{M} \rightarrow \mathcal{N}$, $(\mathcal{G},~g):~\mathcal{N} \rightarrow \mathcal{M}$ and natural isomorphisms
\begin{align}
h:~id_{\mathcal{N}} \rightarrow (\mathcal{F} \circ \mathcal{G}),~\quad k:~id_{\mathcal{M}} \rightarrow (\mathcal{G} \circ \mathcal{F}).
\end{align}

We denote the full subcategory of $Hom_{\mathcal{A}}(\mathcal{M},~\mathcal{N})$ consisting of right exact $\mathcal{A}$ module functors by $Hom^{re}_{\mathcal{A}}(\mathcal{M},~\mathcal{N})$. Similarly, we use $Hom^{le}_{\mathcal{A}}(\mathcal{M},~\mathcal{N})$ to denote the full subcategory of left exact $\mathcal{A}$ module functors and $Hom^e_{\mathcal{A}}(\mathcal{M},~\mathcal{N})$ to denote the full subcategory of exact $\mathcal{A}$ module functors.
\begin{thm}
\label{250}
\cite{etos} Every additive module functor $\mathcal{F}:~\mathcal{M} \rightarrow \mathcal{N}$ between two module categories $\mathcal{M}$ and $\mathcal{N}$ over an FRBSU monoidal category $\mathcal{A}$ is exact if $\mathcal{M}$ is exact. 
\end{thm}

\subsection{Bimodule Category and Some Properties}
\cite{dani} Assume that $\mathcal{A}$ and $\mathcal{B}$ are two finite monoidal categories. A category $\mathcal{M}$ is an $(\mathcal{A}- \mathcal{B})$ bimodule category if it is a left $\mathcal{A}$ module category and right $\mathcal{B}$ module category such that there exists a middle associativity constraint $a$ consisting of a collection of isomorphisms
\begin{align}
\label{3}
a_{XMY}:~X\otimes (M\otimes Y)\rightarrow (X\otimes M)\otimes Y
\end{align}

natural in $X\in \mathcal{A}$, $Y\in \mathcal{B}$, $M\in \mathcal{M}$ which satisfies the commutativity of two pentagons.
\begin{lem}
If $\mathcal{M}$ is an $(\mathcal{A}-\mathcal{B})$ bimodule category, then $\mathcal{M}^{op}$ is a $(\mathcal{B}-\mathcal{A})$ bimodule category.
\end{lem}
\begin{proof}
Assume that $\mathcal{M}$ is an $(\mathcal{A}-\mathcal{B})$ bimodule category. In that situation $\mathcal{M}$ is a left $\mathcal{A}$ module category and a right $\mathcal{B}$ module category, so $\mathcal{M}^{op}$ is a left $\mathcal{B}$ module category and right $\mathcal{A}$ module category by Proposition \ref{1}.\\

We have an associativity constraint $a$ consisting of a family of isomorphisms 
\begin{align*}
a_{XMY}:~X\otimes (M\otimes Y)\rightarrow (X\otimes M)\otimes Y
\end{align*}

natural in $X\in \mathcal{A}$, $Y\in \mathcal{B}$, $M\in \mathcal{M}$ as in \ref{3} which satisfies the commutativity of the required diagrams for all $X,~Y\in \mathcal{A}$, $Z,~W\in \mathcal{B}$ and $M\in \mathcal{M}$ to be an $(\mathcal{A}-\mathcal{B})$ bimodule. We need to define $a^{op}$ consisting of associativity constraints 
\begin{align}
\label{4}
a^{op}_{XMY}:~X\otimes^{op}(M\otimes^{op} Y)\rightarrow (X\otimes^{op} M)\otimes^{op} Y.
\end{align}
  
\ref{4} is obtained by reversing the morphism $Y\otimes (M\otimes X)\rightarrow (Y\otimes M)\otimes X$ in $\mathcal{M}$ which is same as $a_{YMX}$. We can prove the compatibility conditions without no difficulty.
\end{proof}
\begin{lem}
Every finite, rigid monoidal category $\mathcal{A}$ is a bimodule category over itself.
\end{lem}
\begin{proof}
Assume that $\mathcal{A}$ is a finite, rigid monoidal category. We can take $\mathcal{M}=\mathcal{A}$. We have a bifunctor $\mathcal{F}:~\mathcal{A} \times \mathcal{A} \rightarrow \mathcal{A}$ taking $(X,~Y)$ to $X\otimes Y$. $\mathcal{F}$ is exact in each variable by Remark \ref{82}.\\

We can use the associativity constraint $a$ and left unit constraint $l$ in the definition of a monoidal category. We can see that these satisfy the commutativity of the required diagrams to be a left $\mathcal{A}$ module category. Similarly, it is a right $\mathcal{A}$ module category with the same associativity constraint and right unit constraint $r$ by the definition of a monoidal category. Also, we use same associativity constraint and a middle associativity constraint. These satisfy the compatibility conditions.
\end{proof}
\begin{lem}
If $\mathcal{A}$ and $\mathcal{B}$ are finite monoidal categories, then every exact $(\mathcal{A}-\mathcal{B})$ bimodule category $\mathcal{M}$ is finite.
\end{lem}
\begin{lem}
\cite{dani} If $\mathcal{A}$ is a braided monoidal category, then any left $\mathcal{A}$ module category $\mathcal{M}$ is an $(\mathcal{A}-\mathcal{A})$ bimodule category.
\end{lem}
\begin{remark}
$A\mathcal{A}=A\mathcal{A} I$ and $\mathcal{A}B=I\mathcal{A} B$.
\end{remark}
\begin{proof}
Assume that $M$ is an object in $A\mathcal{A}$, so it is a left $A$ module. $M\otimes I\cong M$, so it is a right $I$ module at the same time. As a result, it is an object in $A\mathcal{A} I$. Same for $I\mathcal{A} B$. 
\end{proof}
\begin{lem}
The category $A\mathcal{A} B$ consisting of $(A-B)$ bimodules is an $(\mathcal{A}-\mathcal{A})$ bimodule category.
\end{lem}
\begin{proof}
Every $(A-B)$ bimodule $M$ is a left $A$ module and a right $B$ module in $\mathcal{A}$ that satisfies the compatible conditions, so $M$ is an object in $A\mathcal{A}$ and an object in $\mathcal{A}B$. This means that $A\mathcal{A} B$ is a subcategory of $A\mathcal{A}$ and a subcategory of $\mathcal{A} B$. $A\mathcal{A}$ is a right $\mathcal{A}$ module category and $\mathcal{A} B$ is left $\mathcal{A}$ module category. As a result, $A\mathcal{A} B$ is both a left $\mathcal{A}$ and right $\mathcal{A}$ module category. We need to define an associativity constraint $a$ consisting of associativity isomorphisms as in \ref{3} for all objects $X$, $Y$ in $\mathcal{A}$, $M$ in $A\mathcal{A} B$ that satisfies the required conditions.\\

We have two actions $\mathcal{A} \times A\mathcal{A} B\rightarrow A\mathcal{A} B$ taking $(X,~M)$ to $X\otimes_{l(\mathcal{A} B)} M$ and $A\mathcal{A} B\times \mathcal{A} \rightarrow A\mathcal{A} B$ taking $(M,~Y)$ to $M\otimes_{r(A\mathcal{A})} Y$.\\

We know that $X\otimes_{l(\mathcal{A} B)} M$ is a right $B$ module and we may show that it is a left $A$ module which means that it is an $(A-B)$ bimodule. Similarly, $M\otimes_{r(A\mathcal{A})} Y$ is an $(A-B)$ bimodule.\\

$X\otimes_{l(\mathcal{A}B)}(M\otimes_{r(A\mathcal{A})} Y)\rightarrow (X\otimes_{l(\mathcal{A}B)} M)\otimes_{r(A\mathcal{A})} Y$ is an isomorphism since $M$ is an object in $\mathcal{A}$ and the above actions are exactly same as the tensor product  in $\mathcal{A}$, we can use the associativity constraint in $\mathcal{A}$ as a middle associativity constraint. It is obvious, this gives the commutativity of the diagrams in the definition.
\end{proof}

The following proposition and its proof is found in \cite{dani} and we want to repeat the proof here.
\begin{prop}
\label{10}
The functor $\mathfrak{F}:~A\mathcal{A} B\rightarrow Hom^{re}_{\mathcal{A}}(\mathcal{A}A,~\mathcal{A}B)$ taking $M$ to $-\otimes_A M$ is an equivalence of categories for all algebras $A$ and $B$ in $\mathcal{A}$ for $\mathcal{A}$ is a finite monoidal category. 
\end{prop}
\begin{proof}
We must show that 
\begin{align*}
f:~Hom_{A\mathcal{A} B}(M,~N)\rightarrow Hom_{Hom^{re}_{\mathcal{A}}(\mathcal{A}A,~\mathcal{A}B)}(-\otimes_A M,~-\otimes_A N)
\end{align*}

is an isomorphism and $\mathcal{F}$ is essentially surjective for all $(A-B)$ bimodules $M$ and $N$.\\

We send each morphism $\xymatrix{M\ar[r]^a & N}$ in the category of $(A-B)$ bimodules to a natural transformation $\xymatrix{-\otimes_A M\ar[r]^{f(a)} & -\otimes_A N}$ in the category of module functors $Hom^{re}_{\mathcal{A}}(\mathcal{A}A,~\mathcal{A}B)$. Here, $-\otimes_A M,~-\otimes_A N:~\mathcal{A}A \rightarrow \mathcal{A}B$ are $\mathcal{A}$ module functors for all $(A-B)$ bimodules $M$ and $N$ in $\mathcal{A}$.\\

To show that they are module functors, we show that $K\otimes_A N=K\otimes N$ is a right $B$ module in the category $\mathcal{A}$ at first for all right $A$ modules $K$ in $\mathcal{A}$. We define the action as $a_{(K\otimes N)}=(K\otimes N)\times B\rightarrow (K\otimes N)$ by sending the pair $((k\otimes n),~b)$ to $k\otimes (n\otimes b)$ for all elements $k\in K$, $n\in N$, $b\in B$. As a result, the following diagrams commute since $N$ is a right $B$ module.
\begin{equation*}
\xymatrix@C+1em{ (K\otimes N)\otimes B\otimes B \ar[d]_{a_{(K\otimes N)}\otimes id} \ar[rr]^{id_{(K\otimes N)}\otimes m} & & (K\otimes N)\otimes B \ar[d]^{a_{(K\otimes N)}}\\ (K\otimes N)\otimes B \ar[rr]^{a_{(K\otimes N)}} & & (K\otimes N)} \quad \xymatrix{ (K\otimes N)\otimes B \ar[rr]^{a_{(K\otimes N)}} & & K\otimes N\\ (K\otimes N)\otimes I \ar[u]^{id\otimes u} \ar[urr]_{l_{(K\otimes N)}}}
\end{equation*}

We want to show that $f$ is an injection. Let $a,~b:~M\rightarrow N$ be two $(A-B)$ bimodule homomorphisms and $f(a)=f(b)$.\\

$f(a),~f(b):~-\otimes_A M\rightarrow -\otimes_A N$ are natural transformations and for all objects $K$ in $\mathcal{A}A$, we get $f(a)(K)=f(b)(K):~K\otimes_A M\rightarrow K\otimes_A N$.\\

For all elements $k$ in $K$ and $m$ in $M$, we have $k\otimes a(m)=k\otimes b(m)$, so $k\otimes (a(m)-b(m))=0$ and $a(m)=b(m)$. This says that $a=b$. Surjectivity is clear. As a result $f$ is an isomorphism.\\

For all right exact $\mathcal{A}$ module functors $\mathcal{G}:~\mathcal{A}A \rightarrow \mathcal{A}B$, can we find an $(A-B)$ bimodule $M$ such that $\mathfrak{F}(M) \cong \mathcal{G}$? We take $M=\mathcal{G}(A)$. It is $(A-B)$ bimodule.\\ 

Now we want to show the commutativity of the required diagram for the natural isomorphism. For all morphism $\alpha:~T\rightarrow S$ in $\mathcal{A}A$, we get the following commutative diagram.
\begin{equation*}
\xymatrix{ T\otimes_A \mathcal{G}(A) \ar[rr]^{\cong}_{\varphi} \ar[d]_{\alpha \otimes_A \mathcal{G}(A)} & & \mathcal{G}(T) \ar[d]^{\mathcal{G}(\alpha)}\\ S\otimes_A \mathcal{G}(A) \ar[rr]^{\cong}_{\theta} & & \mathcal{G}(S)}
\end{equation*} 

$T\otimes_A \mathcal{G}(A)\cong T\otimes_A A\cong T\cong \mathcal{G}(T)$ since $\mathcal{G}(A)\cong A$ and $\mathcal{G}(T)\cong T$. Similarly, $\theta$ is an isomorphism, so the diagram commutes. 
\end{proof}

As a result, we see that $Hom^{re}_{\mathcal{A}}(\mathcal{A}A,~\mathcal{A}B)$ is a finite category since the category $A\mathcal{A} B$ is finite.\\

Similarly, the functor $A\mathcal{A}\rightarrow Hom^{le}_{\mathcal{A}}(\mathcal{A}A,~\mathcal{A}B)$ taking $M$ to $Hom_{\mathcal{A}A}(-,~M):~\mathcal{A}A \rightarrow \mathcal{A}B$ is an equivalence of categories for all given algebras $A$ and $B$ in $\mathcal{A}$. So, $Hom^{le}_{\mathcal{A}}(\mathcal{A}A,~\mathcal{A}B)$ is a finite category.
\begin{lem}
$Hom^{re}_{\mathcal{A}}(\mathcal{M},~\mathcal{N})$ is finite if $\mathcal{M}$ and $\mathcal{N}$ are exact module categories over $\mathcal{A}$ and satisfies the required conditions in Proposition \ref{9}.
\end{lem}
\begin{proof}
$\mathcal{M} \simeq \mathcal{A}A$ and $\mathcal{N} \simeq \mathcal{A}B$ for some algebras $A$ and $B$ in $\mathcal{A}$. So, the category $Hom^{re}_{\mathcal{A}}(\mathcal{M},~\mathcal{N})$ is equivalent to the category $A\mathcal{A} B$ and $A\mathcal{A} B$ is finite.
\end{proof}
\begin{lem}
$Hom^{re}_{\mathcal{A}}(\mathcal{M},~\mathcal{M})$ is a monoidal category of endofunctors of $\mathcal{M}$ for $\mathcal{M}$ is a left module category over a finite monoidal category $\mathcal{A}$. 
\end{lem}
\begin{proof}
For given two module categories $\mathcal{F},~\mathcal{G}:~\mathcal{M} \rightarrow \mathcal{M}$, we define their tensor product as the composition $\mathcal{G} \circ \mathcal{F}:~\mathcal{M} \rightarrow \mathcal{M}$ and the unit functor as the identity of $\mathcal{M}$ which is $\mathcal{I}=id_{\mathcal{M}}:~\mathcal{M} \rightarrow \mathcal{M}$. We get an associativity constraint $a$ which is a family of associative isomorphisms ${a_{\mathcal{F} \mathcal{G} \mathcal{H}}:~(\mathcal{F} \circ \mathcal{G}) \circ \mathcal{H} \rightarrow \mathcal{F} \circ (\mathcal{G} \circ \mathcal{H})}$, left and right unit constraints satisfying the required compatibility conditions.
\end{proof}
\begin{prop}
The monoidal category $Hom^{re}_{\mathcal{A}}(\mathcal{A}A,~\mathcal{A} A)$ is strict and rigid. If $\mathcal{A}=Vec(k)$, then it is not a rigid category.
\end{prop}
\begin{proof}
The right and left duals are the right and left adjoint functors. 
\end{proof}
\begin{thm}
If $\mathcal{A}$ is a multifusion category, $\mathcal{M}$ and $\mathcal{N}$ are module categories over $\mathcal{A}$, then the category $Hom^{re}_{\mathcal{A}}(\mathcal{M},~\mathcal{N})$ is semisimple module category over $Hom^{re}_{\mathcal{A}}(\mathcal{M},~\mathcal{M})$ with action given by composition of functors. It is exact if $\mathcal{M}$ and $\mathcal{N}$ are exact module categories over $\mathcal{A}$.
\end{thm} 
\begin{proof}
We define the action as $Hom^{re}_{\mathcal{A}}(\mathcal{M},~\mathcal{M})\times Hom^{re}_{\mathcal{A}}(\mathcal{M},~\mathcal{N})\rightarrow Hom^{re}_{\mathcal{A}}(\mathcal{M},~\mathcal{N})$ with $(\mathcal{F},~\mathcal{G})\mapsto \mathcal{G} \circ \mathcal{F}$.
\end{proof}

\subsection{The Center of A Bimodule Category}
The center $Z_{\mathcal{A}}(\mathcal{M})$ of an $\mathcal{A}$ bimodule category $\mathcal{M}$ is defined in \cite{gr} as below. Here, $\mathcal{A}$ is a finite rigid, monoidal category whose unit object is simple.\\
 
The objects are $(M,~\gamma_M)$ where $M$ is an object in $\mathcal{M}$ and $\gamma_M$ is a family of natural isomorphisms $\gamma_{MX}:~X\otimes M\rightarrow M\otimes X$ which satisfy the commutativity of the following diagram where $X,~Y$ are objects in $\mathcal{A}$ and $M$ is an object in $\mathcal{M}$.
\begin{equation}
\xymatrix{ (X\otimes Y) \otimes M \ar[d]_{a^{-1}_{XYM}} \ar[r]^{\gamma_{M(XY)}} & M\otimes (X\otimes Y) \ar[d]^{a^{-1}_{MXY}}\\ X\otimes (Y\otimes M) \ar[d]_{X\otimes \gamma_{MY}} & (M\otimes X)\otimes Y\\ X\otimes (M\otimes Y)\ar[r]_{a^{-1}_{XMY}} & (X\otimes M)\otimes Y \ar[u]_{\gamma_{MX}\otimes Y}} 
\end{equation}

A morphism between $(M,~\gamma_M)$ and $(N,~\gamma_N)$ in $Z_{\mathcal{A}}(\mathcal{M})$ is a morphism $f:~M\rightarrow N$ in $\mathcal{M}$ satisfying the condition $\gamma_N(X)(id_X\otimes f)=(f\otimes id_X)\gamma_M(X)$.

\subsection{Definition of A Bicategory}
The following definitions are found in \cite{le1} and \cite{le2} in detail.\\

A collection $\mathfrak{X}$ consisting of the objects $A$, $B$, ... is a bicategory if the following conditions are satisfied.
\begin{enumerate}
\item $\mathfrak{X}(A,~B)$ is a category whose objects are 1 arrows $f:~A\rightarrow B$, $g:~A\rightarrow B$, ... and morphisms are 2 arrows $\gamma:~f\Rightarrow g$, $\theta:~f\Rightarrow g$ , ... as shown in the following diagram. 
\begin{align*}
\begin{tikzpicture}[out=145, in=145, relative]
\node (A) at (0,0) {A};
\node (B) at (3,0) {B};
\draw[->, thick, double] (1.5,0.25) -- (1.5,-0.25) [xshift=5pt] node[right, midway] {$\gamma$} (B);
\path[->] (A) edge [bend left] node [above] {f} (B);
\path[->] (A) edge [bend right] node [below] {g} (B);
\end{tikzpicture}
\end{align*}
\item $\mathcal{F}_{ABC}:~\mathfrak{X}(B,~C)\times \mathfrak{X}(A,~B) \rightarrow \mathfrak{X}(A,~C)$ is a functor taking the pairs $(g,~f)$ to $g\circ f=gf$ and $(\theta,~\gamma)$ to $\theta \star \gamma$. $\theta \star \gamma$ is shown as in the following diagram.
\begin{align*}
\begin{tikzpicture}[out=145, in=145, relative]
\node (A) at (0,0) {A};
\node (B) at (3,0) {B};
\node (C) at (6, 0) {C~~=};
\draw[->, thick, double] (1.5,0.25) -- (1.5,-0.25) [xshift=5pt] node[right, midway] {$\gamma$} (B);
\draw[->, thick, double] (4.5,0.25) -- (4.5,-0.25) [xshift=5pt] node[right, midway] {$\theta$} (C);
\path[->] (A) edge [bend left] node [above] {$f$} (B);
\path[->] (A) edge [bend right] node [below] {$g$} (B);
\path[->] (B) edge [bend left] node [above] {$h$} (C);
\path[->] (B) edge [bend right] node [below] {$k$} (C);
\end{tikzpicture}
\begin{tikzpicture}[out=145, in=145, relative]
\node (A) at (0,0) {A};
\node (C) at (3,0) {C};
\draw[->, thick, double] (1.5,0.25) -- (1.5,-0.25) [xshift=5pt] node[right, midway] {$\theta \star \gamma$} (C);
\path[->] (A) edge [bend left] node [above] {$h\circ f$} (C);
\path[->] (A) edge [bend right] node [below] {$k\circ g$} (C);
\end{tikzpicture}
\end{align*}
\item $\mathcal{F}_A:~1\rightarrow \mathfrak{X}(A,~A)$ is a functor sending the object $\star$ in $1$ to the arrow $id_A$ where $1$ is a category with one object.
\item $a_{ABCD}:~\mathcal{F}_{ABD} \circ (\mathcal{F}_{BCD} \times 1) \rightarrow \mathcal{F}_{ACD} \circ (1\times \mathcal{F}_{ABC})$
\begin{align}
\label{45}
\begin{tikzpicture}
\node (A) at (0, 0) {$\mathfrak{X}(C, D) \times \mathfrak{X}(B, C) \times \mathfrak{X}(A, B)$};
\node (C) at (7, 0) {$\mathfrak{X}(B, D) \times \mathfrak{X}(A, B)$};
\node (B) at (0,-2) {$\mathfrak{X}(C, D) \times \mathfrak{X}(A, C)$};
\node (D) at (7, -2) {$\mathfrak{X}(A,~D)$};
\draw[->, thick, double] (4,-0.75)--(3, -1.25) [xshift=5pt] node [right, midway] {$a_{ABCD}$} (B);
\path[->] (A) edge node [above] {$\mathcal{F}_{BCD} \times 1$} (C);
\path[->] (A) edge node [right, midway] {$1\times \mathcal{F}_{ABC}$} (B);
\path[->] (B) edge node [below] {$\mathcal{F}_{ACD}$} (D);
\path[->] (C) edge node [right, midway] {$\mathcal{F}_{ABD}$} (D);
\end{tikzpicture}
\end{align}

is a natural isomorphism. $\xymatrix{a_{ABCD}(f,~g,~h):~(fg)h\ar[rr]^{\sim} & & f(gh)}$ are 2 arrows for all 1 arrows $f:~C\rightarrow D$, $g:~B\rightarrow C$ and $h:~A\rightarrow B$ such that the following pentagon commutes for all 1 arrows $f,~g,~h,~k$.
\begin{align}
\label{46}
\begin{tikzpicture}
\node (A) at (0, 0) {$((fg)h)k$};
\node (B) at (3, 0) {$(f(gh))k$};
\node (C) at (-1, -2) {$(fg)(hk)$};
\node (D) at (4, -2) {$f((gh)k)$};
\node (E) at (1.5, -3.5) {$f(g(hk))$};
\draw [->, thick] (2, -1.5) arc (360:30:10pt);
\path[->] (A) edge node [above] {$a\star id_k$} (B);
\path[->] (A) edge node [right, midway] {$a$} (C);
\path[->] (B) edge node [right, midway] {$a$} (D);
\path[->] (D) edge node [right, midway] {$id_f\star a$} (E);
\path[->] (C) edge node [right, midway] {$a$} (E);
\end{tikzpicture}
\end{align}
\item $r_{AB}:~\mathcal{F}_{AAB} \circ (1\times \mathcal{F}_A) \rightarrow \mathcal{G}$ and $l_{AB}:~\mathcal{F}_{ABB} \circ (\mathcal{F}_B \times 1) \rightarrow \mathcal{H}$ 
\begin{align}
\begin{tikzpicture}
\node (A) at (1, 0) {$\mathfrak{X}(A, B) \times \mathfrak{X}(A, A)$};
\node (B) at (-1, -2) {$\mathfrak{X}(A, B)$};
\node (C) at (3,-2) {$\mathfrak{X}(A, B) \times 1$};
\draw[->, thick, double] (1, -0.5) -- (1, -1.25) [xshift=2pt] node[right] {$r_{AB}$} (B);
\path[->] (A) edge node [left, midway] {$\mathcal{F}_{AAB}$} (B);
\path[->] (C) edge node [right, midway] {$1\times \mathcal{F}_A$} (A);
\path[->] (C) edge node [below] {$\sim$} node [above] {$\mathcal{G}$} (B);
\end{tikzpicture}
\quad
\begin{tikzpicture}
\node (A) at (1, 0) {$\mathfrak{X}(B, B) \times \mathfrak{X}(A, B)$};
\node (B) at (-1, -2) {$\mathfrak{X}(A, B)$};
\node (C) at (3,-2) {$1\times \mathfrak{X}(A, B)$};
\draw[->, thick, double] (1, -0.5) -- (1, -1.25) [xshift=2pt] node[right] {$l_{AB}$} (B);
\path[->] (A) edge node [left, midway] {$\mathcal{F}_{ABB}$} (B);
\path[->] (C) edge node [right, midway] {$\mathcal{F}_B \times 1$} (A);
\path[->] (C) edge node [below] {$\sim$} node [above] {$\mathcal{H}$} (B);
\end{tikzpicture}
\end{align}

are natural isomorphisms. $\xymatrix{r_{AB}(f,~\star):~f\circ id_A \ar[r] & f}$ and $\xymatrix{ l_{AB}(\star,~f):~id_B\circ f\ar[r] & f}$ are 2 arrows for all 1 arrows $f:~A\rightarrow B$ such that the following triangle commutes. 
\begin{align}
\xymatrix{ (fid_A)g\ar[rr]^a \ar[rd]_{r\star id_g} & & f(id_Bg)\ar[ld]^{id_f\star l} \\ & fg}
\end{align}
\end{enumerate}
\begin{remark}
If all natural isomorphisms $a$, $r$, $l$ are identities such that $(fg)h=f(gh)$, $1f=f=f1$ and same conditions are true for the composition of 2 arrows, then $\mathfrak{X}$ is called a 2-category. 
\end{remark}

\section{Internal Hom of Two Objects in A Module Category}
In this section, we are assuming that $\mathcal{M}$ is an exact module category over a finite, rigid monoidal category $\mathcal{A}$ whose unit object $I$ is simple and we are given objects $M$, $N$ in $\mathcal{M}$.
\begin{lem}
\label{75}
The functor $Hom_{\mathcal{M}}(-\otimes M,~N):~\mathcal{A} \rightarrow Set$ is left exact.
\end{lem}
\begin{proof}
Assume that we have an exact sequence $0\rightarrow A\rightarrow B\rightarrow C\rightarrow 0$ for all objects $A$, $B$ and $C$ in $\mathcal{A}$. Then, the sequence $A\otimes M\rightarrow B\otimes M\rightarrow C\otimes M\rightarrow 0$ is exact since $-\otimes M$ is an exact functor by \ref{74}. So, the sequence 
\begin{align*}
0\rightarrow Hom_{\mathcal{M}}(C\otimes M,~N)\rightarrow Hom_{\mathcal{M}}(B\otimes M,~N)\rightarrow Hom_{\mathcal{M}}(A\otimes M,~N)
\end{align*}

is exact since $Hom_{\mathcal{M}}(-,~N)$ is left exact controvariant functor by Example \ref{76}. This proves the left exactness.
\end{proof}
\begin{defn}
Internal hom of $M$ and $N$ is an object $\underline{Hom}_{\mathcal{M}}(M,~N)$ in $\mathcal{A}$ which represents the functor $Hom_{\mathcal{M}}(-\otimes M,~N):~\mathcal{A} \rightarrow Set$ whenever it is representable.
\end{defn}

This means that there exists a natural isomorphism between the functors $Hom_{\mathcal{M}}(-\otimes M,~N)$ and $Hom_{\mathcal{A}}(-,~\underline{Hom}_{\mathcal{M}}(M,~N))$. 
\begin{lem}
The functor $Hom_{\mathcal{M}}(-\otimes M,~N):~\mathcal{A} \rightarrow Set$ is exact if the internal hom $\underline{Hom}_{\mathcal{M}}(M,~N)$ exists and projective in $\mathcal{A}$.
\end{lem}
\begin{proof}
$Hom_{\mathcal{M}}(X\otimes M,~N)\cong Hom_{\mathcal{A}}(X,~\underline{Hom}_{\mathcal{M}}(M,~N))$ for all objects $X$ in $\mathcal{A}$ by existence of representing object. We know that this functor is left exact controvariant functor by Lemma \ref{75}. We want to show that it is right exact. For this, we need to show that the controvariant functor $Hom_{\mathcal{A}}(-,~\underline{Hom}_{\mathcal{M}}(M,~N)):~\mathcal{A} \rightarrow Set$ is right exact which means that the covariant functor  $Hom_{\mathcal{A}^{op}}(\underline{Hom}_{\mathcal{M}}(M,~N),~-):~\mathcal{A}^{op} \rightarrow Set$ is right exact.\\

Assume that $\xymatrix{ 0\ar[r] & X\ar[r]^f & Y\ar[r]^g & Z\ar[r] & 0}$ is an exact sequence in $\mathcal{A}^{op}$. We want to show that the sequence\\

$\xymatrix{0\ar[r] & Hom_{\mathcal{A}^{op}}(\underline{Hom}_{\mathcal{M}}(M,~N),~X)\ar[r]^F & Hom_{\mathcal{A}^{op}}(\underline{Hom}_{\mathcal{M}}(M,~N),~Y)\\
& \ar[r]^G & Hom_{\mathcal{A}^{op}}(\underline{Hom}_{\mathcal{M}}(M,~N),~Z)\ar[r] & 0}$\\

is exact in $Set$. We just need to show that $G$ is an epimorphism since that sequence is left exact. It is obvious since the internal hom is projective.
\end{proof}

This functor is always exact if $\mathcal{A}$ is semisimple by Lemma \ref{77}.
\begin{lem}
$Hom_{\mathcal{M}}(X\otimes M,~N)\cong Hom_{\mathcal{A}}(X,~\underline{Hom}_{\mathcal{M}}(M,~N))$ canonically for all objects $X$ in $\mathcal{A}$ by definition since the functor $Hom_{\mathcal{M}}(-\otimes M,~N)$ is controvariant.
\end{lem}
\begin{lem}
$Hom_{\mathcal{M}}(M,~X\otimes N)\cong Hom_{\mathcal{A}}(I,~X\otimes \underline{Hom}_{\mathcal{M}}(M,~N))$ canonically for all objects $X$ in $\mathcal{A}$.
\end{lem}
\begin{proof}
For all morphisms $M\rightarrow X\otimes N$ in $\mathcal{M}$, we find a morphism 
\begin{align}
X^+\otimes M\rightarrow X^+\otimes X\otimes N\rightarrow I\otimes N=N
\end{align}

by using rigidity of $\mathcal{A}$. Here, we use the evaluation map $ev_{rX}:~X^+\otimes X\rightarrow I$. So, 
\begin{align}
Hom_{\mathcal{M}}(M,~X\otimes N)\cong Hom_{\mathcal{M}}(X^+\otimes M,~N)\cong Hom_{\mathcal{A}}(X^+,~\underline{Hom}_{\mathcal{M}}(M,~N)).
\end{align}

For all morphisms $X^+\rightarrow \underline{Hom}_{\mathcal{M}}(M,~N)$, we get a morphism 
\begin{align}
I\rightarrow X\otimes X^+\rightarrow X\otimes \underline{Hom}_{\mathcal{M}}(M,~N)
\end{align}

by using the coevaluation map. As a result, we get an isomorphism
\begin{align}
Hom_{\mathcal{M}}(X^+,~\underline{Hom}_{\mathcal{M}}(M,~N))\cong Hom_{\mathcal{A}}(I,~X\otimes \underline{Hom}_{\mathcal{M}}(M,~N)).
\end{align}
\end{proof}
\begin{lem}
$\underline{Hom}_{\mathcal{M}}(X\otimes M,~N) \cong \underline{Hom}_{\mathcal{M}}(M,~N) \otimes X^+$ canonically.
\end{lem}
\begin{proof}
\cite{os} We have\\

$Hom_{\mathcal{A}}(K,~\underline{Hom}_{\mathcal{M}}(M,~N)\otimes X^+) \cong Hom_{\mathcal{A}}(K\otimes X,~\underline{Hom}_{\mathcal{M}}(M,~N))\cong \\
Hom_{\mathcal{M}}((K\otimes X)\otimes M,~N)\cong Hom_{\mathcal{M}}(K\otimes (X\otimes M),~N)\cong Hom_{\mathcal{A}}(K,~\underline{Hom}_{\mathcal{M}}(X\otimes M,~N))$\\

for all $K$ in $\mathcal{M}$, so $\underline{Hom}_{\mathcal{M}}(M,~N)\otimes X^+\cong \underline{Hom}_{\mathcal{M}}(X\otimes M,~N)$ canonically.
\end{proof}
\begin{lem}
$\underline{Hom}_{\mathcal{M}}(M,~X\otimes N) \cong X\otimes \underline{Hom}(M,~N)$ canonically.
\end{lem}
\begin{proof}
See \cite{os}.
\end{proof}
\begin{lem}
\cite{etos} If we assume that $\mathcal{A}$ is a braided monoidal category, then $\underline{Hom}_{\mathcal{M}}(M,~M)$ is an algebra for given object $M$ in $\mathcal{M}$ if it exists as a representing object of the functor $Hom_{\mathcal{M}}(-\otimes M,~M)$.
\end{lem}
\begin{proof}
We need to define a multiplication morphism
\begin{align*}
m:~\underline{Hom}_{\mathcal{M}}(M,~M)\otimes \underline{Hom}_{\mathcal{M}}(M,~M)\rightarrow \underline{Hom}_{\mathcal{M}}(M,~M)
\end{align*}

and a unit morphism $u:~I\rightarrow \underline{Hom}_{\mathcal{M}}(M,~M)$ satisfying the required compatibility conditions.\\

\cite{os} finds a multiplication morphism $m$ as below.\\

$id_{\underline{Hom}_{\mathcal{M}}(M,~M)}$ is in $Hom_{\mathcal{A}}(\underline{Hom}_{\mathcal{M}}(M,~M),~\underline{Hom}_{\mathcal{M}}(M,~M))$ since $\underline{Hom}_{\mathcal{M}}(M,~M)$ is an object in $\mathcal{A}$.
\begin{align*}
Hom_{\mathcal{A}}(\underline{Hom}_{\mathcal{M}}(M,~M),~\underline{Hom}_{\mathcal{M}}(M,~M))\cong Hom_{\mathcal{M}}(\underline{Hom}_{\mathcal{M}}(M,~M) \otimes M,~M) 
\end{align*}

by definition. So, we get a unique morphism $f:~\underline{Hom}_{\mathcal{M}}(M,~M)\otimes M\rightarrow M$ corresponding to $id_{\underline{Hom}_{\mathcal{M}}(M,~M)}$. Using this morphism, we get a composition 
\begin{align*}
\begin{tikzpicture}
\node (A) at (0, 0) {$\underline{Hom}_{\mathcal{M}}(M,~M) \otimes (\underline{Hom}_{\mathcal{M}}(M,~M) \otimes M)$};
\node (B) at (7, 0) {$\underline{Hom}_{\mathcal{M}}(M,~M)\otimes M$};
\node (C) at (11, 0) {$M$};
\path[->] (A) edge  node [above] {$id\otimes f$} (B); 
\path[->] (B) edge node [above] {$f$} (C); 
\end{tikzpicture}
\end{align*}

This is same as the morphism $\xymatrix{ (\underline{Hom}_{\mathcal{M}}(M,~M)\otimes \underline{Hom}_{\mathcal{M}}(M,~M))\otimes M\rightarrow M}$ since
\begin{align*}
\underline{Hom}_{\mathcal{M}}(M,~M)\otimes (\underline{Hom}_{\mathcal{M}}(M,~M)\otimes M) \cong (\underline{Hom}_{\mathcal{M}}(M,~M)\otimes \underline{Hom}_{\mathcal{M}}(M,~M))\otimes M.
\end{align*}

This defines a multiplication morphism\\

$\underline{Hom}_{\mathcal{M}}(M,~M) \otimes \underline{Hom}_{\mathcal{M}}(M,~M) \rightarrow \underline{Hom}_{\mathcal{M}}(M,~M)$ as shown in \cite{os}.\\

This multiplication is associative since $\underline{Hom}_{\mathcal{M}}(M,~M)$ is an object in $\mathcal{A}$ and we have the associativity constraint.\\

For the compatibility conditions, we need to show that the following diagrams commute.
\begin{align*}
\begin{tikzpicture}
\node (A) at (0, 0) {$\underline{Hom}_{\mathcal{M}}(M,~M) \otimes \underline{Hom}_{\mathcal{M}}(M,~M) \otimes \underline{Hom}_{\mathcal{M}}(M,~M)$};
\node (B) at (-4, -2.5) {$\underline{Hom}_{\mathcal{M}}(M,~M) \otimes \underline{Hom}_{\mathcal{M}}(M,~M)$};
\node (C) at (4, -2.5) {$\underline{Hom}_{\mathcal{M}}(M,~M) \otimes \underline{Hom}_{\mathcal{M}}(M,~M)$};
\node (D) at (0, -5) {$\underline{Hom}_{\mathcal{M}}(M,~M)$}; 
\path[->] (A) edge  node [left, midway] {$m\otimes id$} (B); 
\path[->] (A) edge node [right, midway] {$id\otimes m$} (C); 
\path[->] (B) edge node [left, midway] {$m$} (D); 
\path[->] (C) edge node [right, midway] {$m$} (D); 
\end{tikzpicture}
\end{align*}
\begin{align*}
\begin{tikzpicture}
\node (A) at (0, 0) {$\underline{Hom}_{\mathcal{M}}(M,~M) \otimes \underline{Hom}_{\mathcal{M}}(M,~M)$};
\node (B) at (7, 0) {$\underline{Hom}_{\mathcal{M}}(M,~M)$};
\node (C) at (0, -2) {$ I\otimes \underline{Hom}_{\mathcal{M}}(M,~M)$};
\path[->] (A) edge  node [above] {$m$} (B); 
\path[->] (C) edge node [left, midway] {$u\otimes id$} (A); 
\path[->] (C) edge node [below] {$l$} (B); 
\end{tikzpicture}
\end{align*}
\begin{align*}
\begin{tikzpicture}
\node (A) at (0, 0) {$\underline{Hom}_{\mathcal{M}}(M,~M) \otimes \underline{Hom}_{\mathcal{M}}(M,~M)$};
\node (B) at (7, 0) {$\underline{Hom}_{\mathcal{M}}(M,~M)$};
\node (C) at (0, -2) {$\underline{Hom}_{\mathcal{M}}(M,~M) \otimes I$};
\path[->] (A) edge  node [above] {$m$} (B); 
\path[->] (C) edge node [left, midway] {$id\otimes u$} (A); 
\path[->] (C) edge node [below] {$r$} (B); 
\end{tikzpicture}
\end{align*}

We have the isomorphisms
\begin{align*}
Hom_{\mathcal{A}}(\underline{Hom}_{\mathcal{M}}(M,~M)\otimes \underline{Hom}_{\mathcal{M}}(M,~M),~\underline{Hom}_{\mathcal{M}}(M,~M)) \cong \\
Hom_{\mathcal{M}}(\underline{Hom}_{\mathcal{M}}(M,~M) \otimes \underline{Hom}_{\mathcal{M}}(M,~M) \otimes M,~M),
\end{align*}
\begin{align*}
Hom_{\mathcal{A}}(\underline{Hom}_{\mathcal{M}}(M,~M),~\underline{Hom}_{\mathcal{M}}(M,~M))\cong Hom_{\mathcal{M}}(\underline{Hom}(M,~M)\otimes M,~M).
\end{align*}

$m$ is a morphism in the first hom set and $id$ is a morphism in the second one.\\ 

Also, we have a composition of hom sets\\

$Hom_{\mathcal{A}}(A\otimes A\otimes A,~A\otimes A)\times Hom_{\mathcal{A}}(A\otimes A,~A)\rightarrow Hom_{\mathcal{A}}(A \otimes A\otimes A,~A)$\\

for $A=\underline{Hom}_{\mathcal{M}}(M,~M)$ taking $(m\otimes id,~m)$ to $m\circ (m\otimes id)$ and $(id\otimes m,~m)$ to $m\circ (id\otimes m)$. We may prove that $m\circ (m\otimes id)=m\circ (id\otimes m)$ for the commutativity of the first diagram.\\ 

Now, we want to find a unit morphism $u:~I\rightarrow \underline{Hom}_{\mathcal{M}}(M,~M)$ satisfying the commutativity of the required diagrams.\\

$Hom_{\mathcal{M}}(M,~X\otimes N)\cong Hom_{\mathcal{A}}(I,~X\otimes \underline{Hom}_{\mathcal{M}}(M,~N))$ for all objects $X$ in $\mathcal{A}$. Taking $X=I$ and $M=N$, we get an isomorphism $Hom_{\mathcal{M}}(M,~M)\cong Hom_{\mathcal{A}}(I,~\underline{Hom}_{\mathcal{M}}(M,~M))$. So, for the identity morphism $id_M:~M\rightarrow M$, we get a unique morphism $u:~I\rightarrow \underline{Hom}_{\mathcal{M}}(M,~M)$.\\

We have an isomorphism
\begin{align*}
Hom_{\mathcal{A}}(I\otimes \underline{Hom}_{\mathcal{M}}(M,~M),~\underline{Hom}_{\mathcal{M}}(M,~M)) \cong Hom_{\mathcal{M}}(I\otimes \underline{Hom}_{\mathcal{M}}(M,~M) \otimes M,~M).
\end{align*}

The diagram
\begin{align*}
\begin{tikzpicture}
\node (A) at (0, 0) {$\underline{Hom}_{\mathcal{M}}(M,~M) \otimes \underline{Hom}_{\mathcal{M}}(M,~M)$};
\node (B) at (7, 0) {$\underline{Hom}_{\mathcal{M}}(M,~M)$};
\node (C) at (0, -2) {$ I\otimes \underline{Hom}_{\mathcal{M}}(M,~M)$};
\path[->] (A) edge  node [above] {$m$} (B); 
\path[->] (C) edge node [left, midway] {$u\otimes id$} (A); 
\path[->] (C) edge node [below] {$l$} (B); 
\end{tikzpicture}
\end{align*}

corresponds to the diagram
\begin{align*}
\begin{tikzpicture}
\node (A) at (0, 0) {$\underline{Hom}_{\mathcal{M}}(M,~M) \otimes \underline{Hom}_{\mathcal{M}}(M,~M)\otimes M$};
\node (B) at (7, 0) {$M$};
\node (C) at (0, -2) {$ I\otimes \underline{Hom}_{\mathcal{M}}(M,~M) \otimes M$};
\path[->] (A) edge  (B); 
\path[->] (C) edge node [left, midway] {$u\otimes id$} (A); 
\path[->] (C) edge (B); 
\end{tikzpicture}
\end{align*}

And we get another diagram
\begin{align*}
\begin{tikzpicture}
\node (A) at (0, 0) {$M\otimes \underline{Hom}_{\mathcal{M}}(M,~M) \otimes M$};
\node (B) at (7, 0) {$M$};
\node (C) at (0, -2) {$M\otimes \underline{Hom}_{\mathcal{M}}(M,~M) \otimes M$};
\path[->] (A) edge  (B); 
\path[->] (C) edge node [left, midway] {$id_M\otimes id$} (A); 
\path[->] (C) edge (B); 
\end{tikzpicture}
\end{align*}

Obviously, this diagram commutes. As a result, the first one commutes and we get the result. We follow the similar way for the right associativity constraint.
\end{proof}
\begin{lem}
\cite{baki} $(^+A)^+=A$ for all objects $A$ in a rigid monoidal category $\mathcal{A}$.
\end{lem}
\begin{lem}
\label{222}
\cite{etos} $\underline{Hom}_{\mathcal{A} A}(M,~M)=(M\otimes ^+M)^+=(^+M)^+\otimes M^+= M\otimes M^+\cong M^+\otimes M$ for all right $A$ modules $M$ in a left module category $\mathcal{A} A$ over an FRBSU monoidal category $\mathcal{A}$.
\end{lem}
\begin{lem}
$\underline{Hom}_{\mathcal{M}}(M,~N)$ is a right $\underline{Hom}_{\mathcal{M}}(M,~M)$ module in $\mathcal{A}$.
\end{lem}
\begin{proof}
We need to define a multiplication morphism 
\begin{align*}
a:~\underline{Hom}_{\mathcal{M}}(M,~N)\otimes \underline{Hom}_{\mathcal{M}}(M,~M) \rightarrow \underline{Hom}_{\mathcal{M}}(M,~N)
\end{align*}

satisfying the commutativity of the following diagrams.
\begin{align*}
\begin{tikzpicture}
\node (A) at (0, 0) {$\underline{Hom}_{\mathcal{M}}(M,~N) \otimes \underline{Hom}_{\mathcal{M}}(M,~M) \otimes \underline{Hom}_{\mathcal{M}}(M,~M)$};
\node (B) at (-4, -2.5) {$\underline{Hom}_{\mathcal{M}}(M,~N) \otimes \underline{Hom}_{\mathcal{M}}(M,~M)$};
\node (C) at (4, -2.5) {$\underline{Hom}_{\mathcal{M}}( M,~N) \otimes \underline{Hom}_{\mathcal{M}}(M,~M)$};
\node (D) at (0, -5) {$\underline{Hom}_{\mathcal{M}}(M,~N)$}; 
\path[->] (A) edge  node [left, midway] {$id\otimes m$} (B); 
\path[->] (A) edge node [right, midway] {$a\otimes id$} (C); 
\path[->] (B) edge node [left, midway] {$a$} (D); 
\path[->] (C) edge node [right, midway] {$a$} (D); 
\end{tikzpicture}
\end{align*}
\begin{align*}
\begin{tikzpicture}
\node (A) at (0, 0) {$\underline{Hom}_{\mathcal{M}}(M,~N) \otimes \underline{Hom}_{\mathcal{M}}(M,~M)$};
\node (B) at (7, 0) {$\underline{Hom}_{\mathcal{M}}(M,~N)$};
\node (C) at (0, -2) {$\underline{Hom}_{\mathcal{M}}(M,~N)\otimes I$};
\path[->] (A) edge  node [above] {$a$} (B); 
\path[->] (C) edge node [left, midway] {$id\otimes u$} (A); 
\path[->] (C) edge node [below] {$l$} (B); 
\end{tikzpicture}
\end{align*}
\begin{align*}
Hom_{\mathcal{A}}(\underline{Hom}_{\mathcal{M}}(M,~N)\otimes \underline{Hom}_{\mathcal{M}}(M,~M),~\underline{Hom}_{\mathcal{M}}(M,~N))\\
\cong Hom_{\mathcal{M}}(\underline{Hom}_{\mathcal{M}}(M,~N)\otimes \underline{Hom}_{\mathcal{M}}(M,~M) \otimes M,~N),
\end{align*}
\begin{align*}
Hom_{\mathcal{A}}(\underline{Hom}_{\mathcal{M}}(M,~N),~\underline{Hom}_{\mathcal{M}}(M,~N))\cong Hom_{\mathcal{M}}(\underline{Hom}_{\mathcal{M}}(M,~N)\otimes M,~N).
\end{align*}

So the identity morphism $id_{\underline{Hom}_{\mathcal{M}}(M,~N)}$ corresponds to a morphism 
\begin{align*}
k:~\underline{Hom}_{\mathcal{M}}(M,~N) \otimes M\rightarrow N
\end{align*}

and $a$ corresponds to a composition
\begin{align*}
\underline{Hom}_{\mathcal{M}}(M,~N) \otimes \underline{Hom}_{\mathcal{M}}(M,~M) \otimes M \rightarrow \underline{Hom}_{\mathcal{M}}(M,~N) \otimes M\rightarrow N
\end{align*}

which is $k\circ (id_{\underline{Hom}_{\mathcal{M}}(M,~N)}\otimes f)$ and $f$ is the morphism $\underline{Hom}_{\mathcal{M}}(M,~M) \otimes M\rightarrow M$ that corresponds to $id_{\underline{Hom}_{\mathcal{M}}(M,~M)}$.\\

It is easy to show that those diagrams commute.
\end{proof}
\begin{lem}
If $A=\underline{Hom}_{\mathcal{M}}(M,~M)$ is the algebra defined as above, then $\mathcal{A} A$ is an exact left module category over $\mathcal{A}$. 
\end{lem}
\begin{prop}
\label{81}
\cite{etos} The mapping $\underline{Hom}_{\mathcal{M}}(M,~-):~\mathcal{M} \rightarrow \mathcal{A}A$ is an exact module functor.
\end{prop}
\begin{proof}
$f$ is a family of natural isomorphisms $f_{XN}:~\underline{Hom}_{\mathcal{M}}(M,~X\otimes N) \rightarrow X\otimes \underline{Hom}_{\mathcal{M}}(M,~N)$ for all objects $X$ in $\mathcal{A}$, $N$ in $\mathcal{M}$ such that for all $X$, $Y$ in $\mathcal{A}$, $N$ in $\mathcal{M}$, the following diagrams commute.
\begin{equation*}
\xymatrix{ \underline{Hom}_{\mathcal{M}}(M,~(X\otimes Y)\otimes N) \ar[d] \ar[r] & \underline{Hom}_{\mathcal{M}}(M,~X\otimes (Y\otimes N)) \ar[r] & X\otimes \underline{Hom}_{\mathcal{M}}(M,~Y\otimes N) \ar[d]\\ (X\otimes Y)\otimes \underline{Hom}_{\mathcal{M}}(M,~N) \ar[rr] & & X\otimes(Y\otimes \underline{Hom}_{\mathcal{M}}(M,~N)) }
\end{equation*}
\begin{equation*}
\xymatrix{ \underline{Hom}_{\mathcal{M}}(M,~I\otimes N) \ar[rd]_{\mathcal{F}(l(M))} \ar[rr] & & I\otimes \underline{Hom}_{\mathcal{M}}(M,~N) \ar[ld] \\ & \underline{Hom}_{\mathcal{M}}(M,~N)}
\end{equation*}

Exactness comes from \ref{250} and this proves the proposition.
\end{proof}
\begin{thm}
\label{9}
\cite{etos} Assume that $\mathcal{M}$ is an exact module category over a finite, rigid monoidal category $\mathcal{A}$ such that the unit object in $\mathcal{A}$ is simple. Let $\underline{Hom}(M,~M)=A$ is the algebra defined as above. Assume further that there exists an object $X\in \mathcal{A}$ for all objects $N\in \mathcal{M}$ such that $Hom_{\mathcal{M}}(X\otimes M,~N)\neq 0$. Then, the functor $\underline{Hom}_{\mathcal{M}}(M,~-):~\mathcal{M} \rightarrow \mathcal{A}A$ is an equivalence of module categories.
\end{thm}
\begin{proof}
We need to show that 
\begin{align*}
Hom_{\mathcal{M}}(N,~K)\rightarrow Hom_{\mathcal{A}A}(\underline{Hom}_{\mathcal{M}}(M,~N),~\underline{Hom}_{\mathcal{M}}(M,~K))
\end{align*}

is an isomorphism for all objects $N$ and $K$ in $\mathcal{M}$ and $\mathcal{F}$ is essentially surjective for the equivalence. \cite{os} proves the isomorphism for all objects $N$ of the form $X\otimes M$ and for all objects $K$ in $\mathcal{M}$ first, then the author proves it for all objects $N$ and $K$ in $\mathcal{M}$ by using the exactness of the functor $\underline{Hom}_{\mathcal{M}}(M,~-)$. After that, \cite{os} shows that $\underline{Hom}_{\mathcal{M}}(M,~-)$ is essentially surjective.\\

We may apply similar way to prove the proposition.  In \cite{os}, the proposition is given whenever the category is semisimple, hence exact means semisimple in such a category. $\mathcal{M}$ is indecomposable module category there instead of the assumption for $\mathcal{M}$ in here.
\end{proof}
\begin{exmp}
$Vec_f(k)$ is an FRBSU monoidal category and it is an exact left module category over itself with the tensor multiplication. Let $A=\underline{Hom}_{Vec_f(k)}(M,~M)$. $A$ is an algebra over $k$ for a right $A$ module $M$ in $Vec_f(k)$. For all right $A$ modules $N$ in $Vec_f(k)$, we get a surjection $N\otimes M\rightarrow N$, hence $M$ genetares $Vec_f(k)$ and as a result, $Vec_f(k)\simeq Vec_f(k)A$ by Theorem \ref{9}.
\end{exmp}
\begin{lem}
\label{79}
The left module category $\mathcal{A} A$ for $A=\underline{Hom}_{\mathcal{M}}(M,~M)$ for some object $M$ in $\mathcal{M}$ is a finite category, hence $\mathcal{M}$ is a finite category.
\end{lem}

\section{ Invertible Bimodule Categories}
We use \cite{gr} and \cite{etnios} at most for the following information.
\begin{lem} 
\label{36}
Assume that $\mathcal{M}$ is an $(\mathcal{A}-\mathcal{B})$ bimodule category and $\mathcal{N}$ is a $(\mathcal{B}-\mathcal{C})$ bimodule category for given finite, rigid monoidal categories $\mathcal{A}$, $\mathcal{B}$ and $\mathcal{C}$. Then, $Hom^{re}_{\mathcal{B}}(\mathcal{M}^{op},~\mathcal{N})$ is an $(\mathcal{A}-\mathcal{C})$ bimodule category. It is abelian. If $\mathcal{A}=\mathcal{B}$, $\mathcal{M}$ and $\mathcal{N}$ are exact, then it is exact. 
\end{lem}
\begin{proof}
To prove that $Hom^{re}_{\mathcal{B}}(\mathcal{M}^{op},~\mathcal{N})$ is an $(\mathcal{A}-\mathcal{C})$ bimodule category, we define the left action of $\mathcal{A}$ on $Hom^{re}_{\mathcal{B}}(\mathcal{M}^{op},~\mathcal{N})$ by 
\begin{align*}
\mathfrak{F}:~\mathcal{A} \times Hom^{re}_{\mathcal{B}}(\mathcal{M}^{op},~\mathcal{N})\rightarrow Hom^{re}_{\mathcal{B}}(\mathcal{M}^{op},~\mathcal{N}),~(A,~\mathcal{F})\rightarrow A\otimes \mathcal{F}
\end{align*}

for all objects $A$ in $\mathcal{A}$ and right exact module functors $\mathcal{F}:~\mathcal{M}^{op} \rightarrow \mathcal{N}$. Here, we have $(A\otimes \mathcal{F})(M)=\mathcal{F}(M\otimes A)$ for all objects $M$ in $\mathcal{M}^{op}$. $\mathcal{M}^{op}$ is a right $\mathcal{A}$ module category, so $M\otimes A$ is an object in $\mathcal{M}^{op}$. We need to show that $\mathfrak{F}$ is a biexact bifunctor, that is $\mathfrak{F}(-,~\mathcal{F})$ is exact for all right exact module functors $\mathcal{F}:~\mathcal{M}^{op} \rightarrow \mathcal{N}$ which means that exact in the first variable and $\mathfrak{F}(A,~-)$ is exact for all objects $A$ in $\mathcal{A}$ which means that exact in the second variable.\\

To prove that $\mathfrak{F}(-,~\mathcal{F})$ is an exact functor, we need to show that 
\begin{align*}
0\rightarrow A\otimes \mathcal{F} \rightarrow B\otimes \mathcal{F} \rightarrow C\otimes \mathcal{F} \rightarrow 0
\end{align*}

is an exact sequence of natural transformations of right exact module functors from $\mathcal{M}^{op}$ to $\mathcal{N}$ whenever the sequence $0\rightarrow A\rightarrow B\rightarrow C\rightarrow 0$ is exact. That sequence of natural transformations is a sequence of morphisms 
\begin{align}
0\rightarrow (A\otimes \mathcal{F})(M) \rightarrow (B\otimes \mathcal{F})(M) \rightarrow (C\otimes \mathcal{F})(M) \rightarrow 0
\end{align}

for all objects $M$ in $\mathcal{M}^{op}$ which satisfies the compatibility conditions for all morphisms $M\rightarrow N$ in $\mathcal{M}^{op}$. This sequence is same as the sequence 
\begin{align}
\label{17}
0\rightarrow \mathcal{F}(M\otimes A) \rightarrow \mathcal{F}(M\otimes B) \rightarrow \mathcal{F}(M\otimes C) \rightarrow 0
\end{align}

by the above action. The sequence $0\rightarrow M\otimes A\rightarrow M\otimes B\rightarrow M\otimes C\rightarrow 0$ is exact since the action is an exact functor by definition of module category. As a result, \ref{17} is an exact sequence by Lemma \ref{250}.\\  

Assume that $0\rightarrow \mathcal{F} \rightarrow \mathcal{G} \rightarrow \mathcal{H} \rightarrow 0$ is an exact sequence of module functors in $Hom^{re}_{\mathcal{B}}(\mathcal{M}^{op},~\mathcal{N})$. It is clear that the sequence $0\rightarrow A\otimes \mathcal{F} \rightarrow A\otimes \mathcal{G} \rightarrow A\otimes \mathcal{H}\rightarrow 0$ is exact. As a result, $\mathfrak{F}(A,~-)$ is an exact functor.\\ 

How do we get an associativity constraint $a$ consisting of a family of associativity isomorphism $a_{AB\mathcal{F}}:~(A\otimes B)\otimes \mathcal{F} \rightarrow A\otimes (B\otimes \mathcal{F})$ for all objects $A$, $B$ in $\mathcal{A}$, right exact module functors $\mathcal{F}:~\mathcal{M}^{op} \rightarrow \mathcal{N}$ and a unit constraint $l$ consisting of a family of unit isomorphisms $l_{\mathcal{F}}:~I\otimes \mathcal{F} \rightarrow \mathcal{F}$ for the unit object $I$ in $\mathcal{A}$, module functor $\mathcal{F}:~\mathcal{M}^{op} \rightarrow \mathcal{N}$ making the required diagrams commute.\\

For all objects $A$, $B$ in $\mathcal{A}$, module functors $\mathcal{F}:~\mathcal{M}^{op} \rightarrow \mathcal{N}$, we get
\begin{align*}
((A\otimes B) \otimes \mathcal{F})(M)=\mathcal{F}(M\otimes (A\otimes B)),\\ (A\otimes (B\otimes \mathcal{F}))(M)=(B\otimes \mathcal{F})(M\otimes A)=\mathcal{F}((M\otimes A)\otimes B).
\end{align*}

$\mathcal{M}^{op}$ is a right $\mathcal{A}$ module category, so $M\otimes (A\otimes B)\cong (M\otimes A)\otimes B)$. We get a short exact sequence $0\rightarrow M\otimes (A\otimes B)\rightarrow (M\otimes A)\otimes B\rightarrow 0$. $\mathcal{F}$ is exact by \ref{250}, so the sequence $0\rightarrow \mathcal{F}(M\otimes (A\otimes B))\rightarrow \mathcal{F}((M\otimes A)\otimes B)\rightarrow 0$ is exact, hence we get an isomorphism $\mathcal{F}(M\otimes (A\otimes B)) \cong \mathcal{F}((M\otimes A)\otimes B)$ and use this isomorphism as an associativity constraint. Similarly, we define $l$.\\ 
 
$\mathfrak{G}:~Hom^{re}_{\mathcal{B}}(\mathcal{M}^{op},~\mathcal{N})\times \mathcal{C} \rightarrow Hom^{re}_{\mathcal{B}}(\mathcal{M}^{op},~\mathcal{N})$ is the right action of $\mathcal{C}$ on $Hom^{re}_{\mathcal{B}}(\mathcal{M}^{op},~\mathcal{N})$ with $(\mathcal{F},~C)\rightarrow \mathcal{F} \otimes C$ for all objects $C$ in $\mathcal{C}$ and right exact module functors $\mathcal{F}:~\mathcal{M}^{op} \rightarrow \mathcal{N}$.\\

Here, $\mathcal{F} \otimes C:~\mathcal{M}^{op} \rightarrow \mathcal{N}$ such that $(\mathcal{F} \otimes C)(M)=\mathcal{F}(M) \otimes C$ for all objects $M$ in $\mathcal{M}^{op}$. $\mathcal{N}$ is right $\mathcal{C}$ module category, so $\mathcal{F}(M) \otimes C$ is an object in $\mathcal{N}$. We can show that $\mathfrak{G}$ is a biexact bifunctor in a similar way. After that we find a right associativity constraint and a right unit constraint satisfying the required conditions.\\

Finally, we find a middle associativity constraint satisfying two commutative diagrams, hence the lemma is proved.
\end{proof}
\begin{cor}
\label{35}
If $\mathcal{M}$ is an $(\mathcal{A}-\mathcal{B})$ bimodule category, then 
\begin{align}
Hom^{re}_{\mathcal{B}}(\mathcal{M}^{op},~\mathcal{B}) \simeq \mathcal{M} \simeq Hom^{re}_{\mathcal{A}}(\mathcal{A}^{op},~\mathcal{M})
\end{align}

canonically as $(\mathcal{A}-\mathcal{B})$ bimodule categories.
\end{cor}
\begin{lem}
\label{40}
\cite{os} $\mathcal{F}_{\mathcal{M}}:~\mathcal{A} \rightarrow Hom^{re}_{\mathcal{A}}(\mathcal{M},~\mathcal{M})$ is a monoidal functor that takes any object $A$ in $\mathcal{A}$ to $A\otimes_{l\mathcal{M}} -$ where $\mathcal{A}$ is a finite, braided monoidal category and $\mathcal{M}$ is a left $\mathcal{A}$ module category.
\end{lem}
\begin{proof}
First, we want to show that $A\otimes -$ is a right exact $\mathcal{A}$ module functor for all objects $A$ in $\mathcal{A}$. The right exactness is clear by definition of module category.\\

Hom sets are vector spaces, because $\mathcal{M}$ is finite by Lemma \ref{79}. That functor is $k$ linear since all functions $f:~Hom_{\mathcal{M}}(A,~B)\rightarrow Hom_{\mathcal{M}}(A\otimes -,~B\otimes -)$ are linear maps for all objects $A$ and $B$ in $\mathcal{M}$.\\

For the functors $\mathcal{F}_1,~\mathcal{F}_2:~A\rightarrow B$, $k_1,~k_2\in k$, we get $f(k_1\mathcal{F}_1+k_2\mathcal{F}_2)=\mathcal{F}$ where $\xymatrix{((k_1+k_2)A)\otimes -\ar[r]^{\mathcal{F}} & ((k_1+k_2)B)\otimes -}$ is a natural transformation. This natural transformation is same as 
\begin{align*}
(k_1(A\otimes -)\rightarrow k_1(B\otimes -))+(k_2(A\otimes -)\rightarrow k_2(B\otimes -))
\end{align*}

which is same as $k_1f(\mathcal{F}_1)+k_2f(\mathcal{F}_2)$.\\

For all objects $B$ in $\mathcal{A}$ and for all objects $M$ in $\mathcal{M}$, we obtain 
\begin{align*}
\xymatrix{f_{BM}=a_{BAM}\circ c_{AB}\circ a^{-1}_{ABM}:~A\otimes (B\otimes M)\ar[rr] & & B\otimes (A\otimes M)}
\end{align*}

by using associativity constraint and the braiding as in the following diagram.
\begin{align*}
\xymatrix{ A\otimes (B\otimes M) \ar[rr]^{\cong}_{a_{ABM}^{-1}} & & (A\otimes B)\otimes M\ar[rr]^{\cong}_{c_{AB}} & & (B\otimes A)\otimes M\ar[rr]^{\cong}_{a_{BAM}} & & B\otimes (A\otimes M)}.
\end{align*}

It is easy to see that the compatibility conditions are satisfied, hence it is a module functor.\\

Now, we want to prove that the assignment $\mathcal{F}_{\mathcal{M}}:~\mathcal{A} \rightarrow Hom^{re}_{\mathcal{A}}(\mathcal{M},~\mathcal{M})$ taking any object $A$ in $\mathcal{A}$ to a right exact module functor $\mathcal{F}(A):~\mathcal{M} \rightarrow \mathcal{M}$ defined by $\mathcal{F}(A)(M)=A\otimes M$ for all objects $M$ in $\mathcal{M}$ is a monoidal functor.\\

There exists a natural transformation $\gamma_{AB}:~\mathcal{F}_{\mathcal{M}}(A) \circ \mathcal{F}_{\mathcal{M}}(B) \rightarrow \mathcal{F}_{\mathcal{M}}(A\otimes B)$.\\
 
$(\mathcal{F}_{\mathcal{M}}(A)\circ \mathcal{F}_{\mathcal{M}}(B))(M)=\mathcal{F}_{\mathcal{M}}(A)(B\otimes M)=A\otimes (B\otimes M)$ and $\mathcal{F}_{\mathcal{M}}(A\otimes B)(M)=(A\otimes B)\otimes M$ for all $M$ in $\mathcal{M}$.\\

We have an isomorphism $a^{-1}:~A\otimes (B\otimes M)\rightarrow (A\otimes B)\otimes M$. This says that 
\begin{align*}
a^{-1}:~(\mathcal{F}_{\mathcal{M}}(A) \circ \mathcal{F}_{\mathcal{M}}(B))(M)\rightarrow \mathcal{F}_{\mathcal{M}}(A\otimes B)(M)
\end{align*}

is an isomorphism for all $M$ in $\mathcal{M}$, so $\gamma_{AB}$ is a natural isomorphism.\\

Also $I\rightarrow \mathcal{F}_{\mathcal{M}}(I)=I\otimes M$ is an isomorphism by left unit constraint. We may show that the diagrams commute. So, we get the result.
\end{proof}
\begin{defn}
An $(\mathcal{A}-\mathcal{B})$ bimodule category $\mathcal{M}$ for $\mathcal{A}$ and $\mathcal{B}$ are finite monoidal categories is invertible if the monoidal functors $\mathcal{B} \rightarrow Hom^{re}_{\mathcal{A}}(\mathcal{M},~\mathcal{M})$ taking all objects $B$ in $\mathcal{B}$ to $-\otimes B$ and $\mathcal{A} \rightarrow Hom^{re}_{\mathcal{B}}(\mathcal{M}^{op},~\mathcal{M}^{op})=Hom_{\mathcal{B}}(\mathcal{M},~\mathcal{M})$ taking all objects $A$ in $\mathcal{A}$ to $A\otimes -$ are equivalences as bimodule categories.
\end{defn}

The following proposition is found in \cite{etnios} for fusion categories.
\begin{prop}
\label{41}
An $(\mathcal{A}-\mathcal{B})$ bimodule category $\mathcal{M}$ for given finite monoidal categories $\mathcal{A}$ and $\mathcal{B}$ is invertible if and only if for all objects $A$ in $\mathcal{A}$ and $B$ in $\mathcal{B}$, the monoidal functor 
\begin{align*}
\mathcal{R}:~\mathcal{B}^{rev} \rightarrow Hom^{re}_{\mathcal{A}}(\mathcal{M},~\mathcal{M}),~\mathcal{R}(B)(M)=M\otimes B
\end{align*}

is an equivalence of $(\mathcal{B}-\mathcal{B})$ bimodule categories if and only if the monoidal functor 
\begin{align*}
\mathcal{L}:~\mathcal{A} \rightarrow Hom^{re}_{\mathcal{B}}(\mathcal{M},~\mathcal{M}),~L(A)(M)=A\otimes M
\end{align*}

is an equivalence of $(\mathcal{A}-\mathcal{A})$ bimodule categories. 
\end{prop}
\begin{proof}
If $\mathcal{M}$ is invertible, then $\mathcal{B} \simeq Hom^{re}_{\mathcal{A}}(\mathcal{M},~\mathcal{M})$ as $(\mathcal{B}-\mathcal{B})$ bimodule categories. So, $\mathcal{B}^{rev} \simeq Hom^{re}_{\mathcal{A}}(\mathcal{M},~\mathcal{M})$.\\

Similarly, $\mathcal{A} \simeq Hom^{re}_{\mathcal{B}}(\mathcal{M}^{op},~\mathcal{M}^{op})=Hom^{re}_{\mathcal{B}}(\mathcal{M},~\mathcal{M})$ as $(\mathcal{B}-\mathcal{B})$ bimodule categories.\\

Conversely, if $\mathcal{R}$ is an equivalence, then $\mathcal{B}^{rev} \simeq Hom^{re}_{\mathcal{A}}(\mathcal{M},~\mathcal{M})$ as $(\mathcal{B}-\mathcal{B})$ bimodule categories, so $\mathcal{B} \simeq Hom^{re}_{\mathcal{A}}(\mathcal{M},~\mathcal{M})$ and if $\mathcal{L}$ is an equivalence, then 
\begin{align*}
\mathcal{A} \simeq Hom^{re}_{\mathcal{B}}(\mathcal{M},~\mathcal{M})=Hom^{re}_{\mathcal{B}}(\mathcal{M}^{op},~\mathcal{M}^{op})
\end{align*}

as $(\mathcal{A}-\mathcal{A})$ bimodule categories. If $\mathcal{R}$ and $\mathcal{L}$ are equivalences, then $\mathcal{M}$ is invertible.\\

See \cite{etnios} for the rest of the proof.
\end{proof}
\begin{note}
If $\mathcal{A}$ is a finite braided monoidal category and $\mathcal{M}$ is an invertible left $\mathcal{A}$ module category such that $\mathcal{M} \simeq \mathcal{A} A$, then $\mathcal{M}$ is an $(\mathcal{A}-\mathcal{A})$ bimodule category with right action given by $M\otimes_{r\mathcal{M}} A=A\otimes_{l\mathcal{M}} M$ for all objects $A$ in $\mathcal{A}$ and $M$ in $\mathcal{M}$.
\end{note}
\begin{remark}
\label{38}
If $\mathcal{A}$ is a finite braided monoidal category, $\mathcal{M}$ is an invertible $(\mathcal{A}-\mathcal{A})$ bimodule category and $A$ is an object in $\mathcal{A}$, then we obtain two monoidal equivalences 
\begin{align*}
\mathcal{R}:~\mathcal{A}^{rev} \rightarrow Hom^{re}_{\mathcal{A}}(\mathcal{M},~\mathcal{M}),~\mathcal{R}(A)(M)=M\otimes A\\
\mathcal{L}:~\mathcal{A} \rightarrow Hom^{re}_{\mathcal{A}}(\mathcal{M},~\mathcal{M}),~\mathcal{L}(A)(M)=A\otimes M
\end{align*}

by Proposition \ref{41} for all objects $M$ in $\mathcal{M}$.
\end{remark}
\begin{cor}
\cite{etnios} Assume that the left module category $\mathcal{M}$ over a fusion category $\mathcal{A}$ is invertible. Then, it is indecomposable left module category over $\mathcal{A}$.
\end{cor}
\begin{proof}
Let $\mathcal{A}$ be a fusion category and $\mathcal{M}$ be decomposable module category over $\mathcal{A}$ under the given conditions. Then, $\mathcal{M} \simeq \mathcal{P} \oplus \mathcal{Q}$ for indecomposable module categories $\mathcal{P}$ and $\mathcal{Q}$ over $\mathcal{A}$ and $A=\underset{i}{\oplus} A_i$ for simple objects $A_i$ in $\mathcal{A}$ for all objects $A$ in $\mathcal{A}$ since $\mathcal{A}$ is a semisimple monoidal category. The monoidal functor $L:~\mathcal{A} \rightarrow Hom^{re}_{\mathcal{B}}(\mathcal{M},~\mathcal{M}),~L(A)(M)=A\otimes M$ between monoidal categories is an equivalence for all objects $A$ in $\mathcal{A}$ and for all objects $M$ in $\mathcal{M}$ under these conditions.\\

$A\otimes M=(\underset{i}{\oplus} A_i)\otimes_{l\mathcal{M}} M\simeq (\underset{i}{\oplus} A_i)\otimes_{l\mathcal{M}} (P\oplus Q)=((\underset{i}{\oplus} A_i)\otimes_{l\mathcal{P}} P)\oplus ((\underset{i}{\oplus} A_i)\otimes_{l\mathcal{Q}} Q)\\
=(\underset{i}{\oplus}(A_i\otimes_{l\mathcal{P}} P) \oplus (\underset{i}{\oplus}(A_i\otimes_{l\mathcal{Q}} Q))$ for all objects $P$ in $\mathcal{P}$ and $Q$ in $\mathcal{Q}$.\\

However, $\underset{i}{\oplus}(A_i\otimes_{l\mathcal{P}} P) \oplus \underset{i}{\oplus}(A_i\otimes_{l\mathcal{Q}} Q)$ is an object in $\mathcal{P} \oplus \mathcal{Q}$. This means that $\underset{i}{\oplus}(A_i\otimes_{l\mathcal{P}} P)$ is an object in $\mathcal{P}$ and $\underset{i}{\oplus}(A_i\otimes_{l\mathcal{Q}} Q)$ is an object in $\mathcal{Q}$. So, $\mathcal{P}=\underset{i}{\oplus}(A_i\otimes \mathcal{P})$ and $\mathcal{Q}=\underset{i}{\oplus}(A_i\otimes \mathcal{Q})$. This is a contradiction since $\mathcal{P}$ and $\mathcal{Q}$ are indecomposable module categories over $\mathcal{A}$. As a result, $\mathcal{M}$ is indecomposable.
\end{proof}

\section{2-Category of FRBSU Monoidal Categories}
\begin{lem}
\label{47}
The composition of two monoidal functors is again a monoidal functor. 
\end{lem}
\begin{proof}
There exists a functor $Hom(\mathcal{B},~\mathcal{C})\times Hom(\mathcal{A},~\mathcal{B})\rightarrow Hom(\mathcal{A},~\mathcal{C})$ taking any pair $(\mathcal{F},~\mathcal{G})$ to $\mathcal{F} \circ \mathcal{G}$ for given monoidal functors $(\mathcal{F},~\beta,~\varphi_1):~\mathcal{B} \rightarrow \mathcal{C}$ in $Hom(\mathcal{B},~\mathcal{C})$ and $(\mathcal{G},~\psi,~\varphi_2):~\mathcal{A} \rightarrow \mathcal{B}$ in $Hom(\mathcal{A},~\mathcal{B})$.\\

$\beta$ is a family of natural isomorphisms $\beta_{XY}:~\mathcal{F}(X) \otimes \mathcal{F}(Y) \rightarrow \mathcal{F}(X\otimes Y)$ for all objects $X$, $Y$ in $\mathcal{B}$ and $\psi$ is a family of natural isomorphisms $\psi_{AB}:~\mathcal{G}(A) \otimes \mathcal{G}(B) \rightarrow \mathcal{G}(A\otimes B)$ for all objects $A$ and $B$ in $\mathcal{A}$.\\

We want to show that $\mathcal{F} \circ \mathcal{G}$ is a monoidal category. We define $\gamma$ as a family of natural isomorphisms $\gamma_{AB}=\mathcal{F}(\psi_{AB}) \circ \beta_{\mathcal{G}(A) \mathcal{G}(B)}$ for all objects $A$ and $B$ in $\mathcal{A}$ as in the following diagram.
\begin{align*}
\begin{tikzpicture}
\node (A) at (0,0) {$\mathcal{F}(\mathcal{G}(A)) \otimes\mathcal{F}(\mathcal{G}(B))$};
\node (B) at (6,0) {$\mathcal{F}(\mathcal{G}(A) \otimes \mathcal{G}(B))$};
\node (C) at (12,0) {$\mathcal{F}(\mathcal{G}(A\otimes B))$};
\node (D) at (0,-2) {$(\mathcal{F} \circ \mathcal{G})(A)\otimes (\mathcal{F} \circ \mathcal{G})(B)$};
\node (E) at (6, -2) {$\mathcal{F}(\mathcal{G}(A) \otimes \mathcal{G}(B))$};
\node (F) at (12,-2) {$(\mathcal{F} \circ \mathcal{G})(A\otimes B)$};
\draw[thick, double] (0, -0.2)--(0, -1.8) [xshift=5pt];
\draw[thick, double] (12, -0.2)--(12, -1.8) [xshift=5pt];
\path[->] (A) edge [right] node [above] {$\beta_{\mathcal{G}(A) \mathcal{G}(B)}$} (B);
\path[->] (B) edge [right] node [above] {$\mathcal{F}(\psi_{AB})$} (C);
\path[->] (D) edge [right] node [above] {$\beta_{\mathcal{G}(A) \mathcal{G}(B)}$} (E);
\path[->] (E) edge [right] node [above] {$\mathcal{F}(\psi_{AB})$} (F);
\end{tikzpicture}
\end{align*}  

$I\cong \mathcal{G}(I)$, so $I\cong \mathcal{F}(I) \cong (\mathcal{F} \circ \mathcal{G})(I)$ for the unit object $I$ in $\mathcal{A}$. It is easy to prove the commutativity of the required diagrams.\\

If we have natural transformations $\theta_1:~\mathcal{F}_1\Rightarrow \mathcal{F}_2$ in $Hom(\mathcal{B},~\mathcal{C})$ and $\theta_2:~\mathcal{G}_1 \Rightarrow \mathcal{G}_2$ in $Hom(\mathcal{A},~\mathcal{B})$, then the composition is again a natural transformation. This gives a morphism in $Hom(\mathcal{A},~\mathcal{C})$ corresponding to the pair $(\theta_1,~\theta_2)$ in $Hom(\mathcal{B},~\mathcal{C})\times Hom(\mathcal{A},~\mathcal{B})$.
\end{proof}
\begin{prop}
The collection of all FRBSU monoidal categories forms a 2-category $\mathcal{MCT}$ in which 1 arrows are braided monoidal functors of those categories and 2 arrows are natural isomorphisms between those monoidal functors.
\end{prop}
\begin{proof}
$\mathfrak{X}=\mathcal{MCT}$. Objects are FRBSU monoidal categories.\\

$\mathcal{MCT}(\mathcal{A},~\mathcal{B})$ is a category in which 1 arrows are braided monoidal functors $\mathcal{K}:~\mathcal{A} \rightarrow \mathcal{B}$ and 2 arrows are natural isomorphisms $\theta:~\mathcal{K} \Rightarrow \mathcal{L}$ for all monoidal functors $\mathcal{K},~\mathcal{L}:~\mathcal{A} \rightarrow \mathcal{B}$. We show 2 arrows as in the following diagram.
\begin{align*}
\begin{tikzpicture}[out=145, in=145, relative]
\node (A) at (0,0) {$\mathcal{A}$};
\node (B) at (3,0) {$\mathcal{B}$};
\draw[->, thick, double] (1.5,0.25) -- (1.5,-0.25) [xshift=5pt] node[right, midway] {$\theta$} (B);
\path[->] (A) edge [bend left] node [above] {$\mathcal{K}$} (B);
\path[->] (A) edge [bend right] node [below] {$\mathcal{L}$} (B);
\end{tikzpicture}
\end{align*}

$\mathcal{F}_{\mathcal{A} \mathcal{B} \mathcal{C}}:~\mathcal{MCT}(\mathcal{B},~\mathcal{C}) \times \mathcal{MCT}(\mathcal{A},~\mathcal{B}) \rightarrow \mathcal{MCT}(\mathcal{A},~\mathcal{C})$ is a functor taking the pair $(\mathcal{L},~\mathcal{K})$ to $\mathcal{L} \circ \mathcal{K}$ and the pair $(\theta,~\gamma)$ to $\theta \star \gamma$ by Lemma \ref{47}.\\

$\mathcal{F}_{\mathcal{A}}:~1\rightarrow \mathcal{MCT}(\mathcal{A},~\mathcal{A})$ is a functor sending the object $\star$ in $1$ to 1 arrow $id_{\mathcal{A}}:~\mathcal{A}\rightarrow \mathcal{A}$.
\begin{align*}
a_{\mathcal{A} \mathcal{B} \mathcal{C} \mathcal{D}}:~\mathcal{F}_{\mathcal{A} \mathcal{B} \mathcal{D}} \circ (\mathcal{F}_{\mathcal{B} \mathcal{C} \mathcal{D}} \times 1) \rightarrow \mathcal{F}_{\mathcal{A} \mathcal{C} \mathcal{D}} \circ (1\times \mathcal{F}_{\mathcal{A} \mathcal{B} \mathcal{C}})
\end{align*}
\begin{align*}
\begin{tikzpicture}
\node (A) at (0, 0) {$\mathcal{MCT}(\mathcal{C},~\mathcal{D}) \times \mathcal{MCT}(\mathcal{B},~\mathcal{C}) \times \mathcal{MCT}(\mathcal{A},~\mathcal{B})$};
\node (C) at (8, 0) {$\mathcal{MCT}(\mathcal{B},~\mathcal{D}) \times \mathcal{MCT}(\mathcal{A},~\mathcal{B})$};
\node (B) at (0,-2) {$\mathcal{MCT}(\mathcal{C},~\mathcal{D})\times \mathcal{MCT}(\mathcal{A},~\mathcal{C})$};
\node (D) at (8, -2) {$\mathcal{MCT}(\mathcal{A},~\mathcal{D})$};
\draw[->, thick, double] (4,-0.75)--(3, -1.25) [xshift=5pt] node [right, midway] {$a_{\mathcal{A} \mathcal{B} \mathcal{C} \mathcal{D}}$} (B);
\path[->] (A) edge node [above] {$\mathcal{F}_{\mathcal{B} \mathcal{C} \mathcal{D}} \times 1$} (C);
\path[->] (A) edge node [right, midway] {$1\times \mathcal{F}_{\mathcal{A} \mathcal{B} \mathcal{C}}$} (B);
\path[->] (B) edge node [below] {$\mathcal{F}_{\mathcal{A} \mathcal{C} \mathcal{D}}$} (D);
\path[->] (C) edge node [right, midway] {$\mathcal{F}_{\mathcal{A} \mathcal{B} \mathcal{D}}$} (D);
\end{tikzpicture}
\end{align*}

is a natural isomorphism.
\begin{align*}
(\mathcal{F}_{\mathcal{A} \mathcal{B} \mathcal{D}} \circ (\mathcal{F}_{\mathcal{B} \mathcal{C} \mathcal{D}} \times 1))(\mathcal{K},~\mathcal{L},~\mathcal{M})=\mathcal{F}_{\mathcal{A} \mathcal{B} \mathcal{D}}(\mathcal{K} \circ \mathcal{L},~\mathcal{M})=(\mathcal{K}\circ \mathcal{L})\circ \mathcal{M},
\end{align*}
\begin{align*}
(\mathcal{F}_{\mathcal{A} \mathcal{C} \mathcal{D}} \circ (1\times \mathcal{F}_{\mathcal{A} \mathcal{B} \mathcal{C}}))(\mathcal{K},~\mathcal{L},~\mathcal{M})=\mathcal{F}_{\mathcal{A} \mathcal{C} \mathcal{D}}(\mathcal{K},~\mathcal{L}\circ \mathcal{M})=\mathcal{K}\circ (\mathcal{L}\circ \mathcal{M})
\end{align*}

for all 1 arrows $\mathcal{K},~\mathcal{L},~\mathcal{M}$. $(\mathcal{K}\circ \mathcal{L})\circ \mathcal{M}=\mathcal{K}\circ (\mathcal{L}\circ \mathcal{M})$ and $a_{\mathcal{A} \mathcal{B} \mathcal{C} \mathcal{D}}(\mathcal{K},~\mathcal{L},~\mathcal{M})=id_{\mathcal{K} \circ \mathcal{L}\circ \mathcal{M}}$.\\

For all morphisms $\alpha:~(\mathcal{K}_1,~\mathcal{L}_1,~\mathcal{M}_1)\rightarrow (\mathcal{K}_2,~\mathcal{L}_2,~\mathcal{M}_2)$ in 
\begin{align*}
\mathcal{MCT}(\mathcal{C},~\mathcal{D}) \times \mathcal{MCT}(\mathcal{B},~\mathcal{C}) \times \mathcal{MCT}(\mathcal{A},~\mathcal{B}),
\end{align*}

the following diagram 
\begin{align*}
\begin{tikzpicture}
\node (A) at (0, 0) {$(\mathcal{K}_1 \circ \mathcal{L}_1) \circ \mathcal{M}_1$};
\node (C) at (8, 0) {$(\mathcal{K}_2 \circ \mathcal{L}_2) \circ \mathcal{M}_2$};
\node (B) at (0,-2) {$\mathcal{K}_1 \circ (\mathcal{L}_1 \circ \mathcal{M}_1)$};
\node (D) at (8, -2) {$\mathcal{K}_2 \circ (\mathcal{L}_2 \circ \mathcal{M}_2)$};
\path[->] (A) edge node [above] {$(\mathcal{F}_{\mathcal{A} \mathcal{B} \mathcal{D}} \circ (\mathcal{F}_{\mathcal{B} \mathcal{C} \mathcal{D}} \times 1))(\alpha)$} (C);
\path[->] (A) edge node [left, midway] {$a_{\mathcal{A} \mathcal{B} \mathcal{C} \mathcal{D}}(\mathcal{K}_1,~\mathcal{L}_1,~\mathcal{M}_1)$} (B);
\path[->] (B) edge node [below] {$(\mathcal{F}_{\mathcal{A} \mathcal{C} \mathcal{D}} \circ (1\times \mathcal{F}_{\mathcal{A} \mathcal{B} \mathcal{C}}))(\alpha)$} (D);
\path[->] (C) edge node [right, midway] {$a_{\mathcal{A} \mathcal{B} \mathcal{C} \mathcal{D}}(\mathcal{K}_2,~\mathcal{L}_2,~\mathcal{M}_2)$} (D);
\end{tikzpicture}
\end{align*}

commutes. As a result $a_{\mathcal{A} \mathcal{B} \mathcal{C} \mathcal{D}}(\mathcal{K},~\mathcal{L},~\mathcal{M})$ is a natural isomorphism for all 1 arrows $\mathcal{K}$, $\mathcal{L}$ and $\mathcal{M}$. Similarly, we show $r$ and $l$ are natural isomorphisms. Hence, $MCT$ is a 2-category.
\end{proof}

\section{Crossed Modules and Morphisms Of Crossed Modules}
A crossed module $\xymatrix{\mathfrak{C}=[N\ar[r]^{h_{\mathfrak{C}}} & M]}$ is a pair of groups $(M,~N)$ such that $M$ acts on $N$ by $M\times N\rightarrow N$ taking $(m,~n)$ to $^mn$ and $h_{\mathfrak{C}}:~N \rightarrow M$ is a group homomorphism satisfying the conditions $h_{\mathfrak{C}}(^m n)=mh_{\mathfrak{C}}(n)m^{-1}$ and $^{h_{\mathfrak{C}}(n)} n'=nn'n^{-1}$ for all $n,~n'\in N$ and $m\in M$.

\subsection{Strict Morphisms and Butterflies Between Crossed Modules}
We use \cite{no} and \cite{alno} as references here.
\begin{defn}
For given crossed modules $\xymatrix{\mathfrak{C_1}=[N_1\ar[r]^{h_{\mathfrak{C_1}}} & M_1]}$ and $\xymatrix{\mathfrak{C_2}=[N_2\ar[r]^{h_{\mathfrak{C_2}}} & M_2]}$, a strict morphism $F=(f_1,~f_2)$ between them is a pair of group homomorphisms $f_1:~M_1 \rightarrow M_2$ and $f_2:~N_1\rightarrow N_2$ such that $h_{\mathfrak{C_2}} \circ f_2=f_1\circ h_{\mathfrak{C_1}}$ and $f_2(^mn)=^{f_1(m)}f_2(n)$ for all $n\in N_1$ and $m\in M_1$. We show that morphism by the following diagram. 
\begin{align}
\label{24}
\begin{tikzpicture}
\node (A) at (0, 0) {$N_1$};
\node (B) at (2, 0) {$N_2$};
\node (C) at (0, -2) {$M_1$};
\node (D) at (2, -2) {$M_2$};
\path[->] (A) edge node [above] {$f_2$} (B);
\path[->] (A) edge node [left] {$h_{\mathfrak{C_1}}$} (C);
\path[->] (B) edge node [right] {$h_{\mathfrak{C_2}}$} (D);
\path[->] (C) edge node [below] {$f_1$} (D);
\end{tikzpicture}
\end{align}
\end{defn}
\begin{defn}
The strict morphism in Diagram \ref{24} is an equivalence of crossed modules if $\pi_2(\mathfrak{C_1})\cong \pi_2(\mathfrak{C_2})$ and $\pi_1(\mathfrak{C_1})\cong \pi_1(\mathfrak{C_2})$.
\end{defn}
\begin{defn}
The commutative diagram of group homomorphisms 
\begin{equation}
\label{23}
\xymatrix{N_1\ar[rd]^f \ar[dd]_{h_{\mathfrak{C_1}}} & & N_2\ar[ld]_k \ar[dd]^{h_{\mathfrak{C_2}}} \\ & E\ar[rd]_t \ar[ld]^g & \\ M_1 & & M_2} 
\end{equation} 

is a butterfly between two crossed modules $\xymatrix{\mathfrak{C_1}=[N_1\ar[r]^{h_{\mathfrak{C_1}}} & M_1]}$ and $\xymatrix{\mathfrak{C_2}=[N_2\ar[r]^{h_{\mathfrak{C_2}}} & M_2]}$ if it satisfies the following axioms. 
\begin{enumerate}
\item Both diagonal sequences are complexes,
\item The NE-SW sequence is a group extension,
\item $k(^{t(x)}n_2)=xk(n_2)x^{-1}$ and $f(^{g(x)}n_1)=xf(n_1)x^{-1}$ for all $x\in E,~n_1\in N_1,~n_2\in N_2$.
\end{enumerate}
\end{defn}

We denote the above butterfly with $(E,~t,~g,~k,~f)$. Also, we can denote that butterfly by $P:~\mathfrak{C_1} \rightarrow \mathfrak{C}_2$ for crossed modules $\mathfrak{C_1}$ and $\mathfrak{C}_2$.
\begin{defn}
A butterfly is reversible(equivalence) if both of the diagonals are extensions, that is a butterfly such that the NW-SE sequence is short exact. It is splittable if there exists a splitting homomorphism $s:~M_1\rightarrow E$ such that $g\circ s=id_{M_1}$ which is same as the condition that the NE-SW sequence is a split extension.
\end{defn}

The inverse of the butterfly in Diagram \ref{23} is shown as in the following diagram which is a butterfly.
\begin{equation}
\xymatrix{N_2\ar[rd]^k \ar[dd]_{h_{\mathfrak{C_2}}} & & N_1\ar[ld]_f \ar[dd]^{h_{\mathfrak{C_1}}} \\ & E\ar[rd]_g \ar[ld]^t & \\ M_2 & & M_1} 
\end{equation} 
\begin{prop}
\label{29}
Every split butterfly corresponds to a unique strict morphism $(f_1,~f_2)$ between two crossed modules.
\end{prop}
\begin{proof}
Assume that $(f_1,~f_2)$ is a strict morphism as in Diagram \ref{24}. We get a commutative diagram 
\begin{equation}
\label{25}
\xymatrix{ N_1 \ar[rd]^f \ar[dd]_{h_{\mathfrak{C}_1}} & & N_2 \ar[ld]_k \ar[dd]^{h_{\mathfrak{C}_2}}\\ & N_2\ltimes M_1\ar[rd]_t \ar[ld]^g & \\ M_1 & & M_2}
\end{equation} 

which means that the NE-SW sequence is a split extension in that butterfly, that is $E=N_2\ltimes M_1$ with the product law $(n_1,~m_1).(n_2,~m_2)=(n_1.^{f_1(m_1)}n_2,~m_1.m_2)$ for all $m_1,~m_2\in M_1$ and $n_1,~n_2\in N_2$.\\

Here, we define $g$ as a projection, $k(n)=(id_{N_2},~1)(n)=(n,~1)$ for all $n\in N_2$, $f(n)=(f_2( n^{-1}),~h_1(n))$ for all $n\in N_1$ and $t(n,~m)=h_2(n).f_1(m)$ for all $n\in N_2$ and $m\in M_1$.\\

Conversely, if we are given a split butterfly as in Diagram \ref{25}, we can find a canonical splitting homomorphism $s:~M_1\rightarrow N_2\ltimes M_1$ taking $m$ to $(1,~m)$ for all $m\in M_1$. We define $f_1=t\circ s$. $f_2$ can be defined from the equation $s\circ h_1=f.(k\circ f_2)$. We can see that those group homomorphisms satisfy the required conditions. 
\end{proof}

\subsection{Strict Morphisms and 2-Category of Crossed Modules}
\begin{lem}
\cite{no} The collection $\mathcal{XM}$ consisting of crossed modules forms a category whose morphisms are strict morphisms of crossed modules as defined in \ref{24}.
\end{lem} 
\begin{defn}
A pointed natural transformation $PNT:~G\Rightarrow F$ between two strict morphisms $F=(f_1,~f_2)$ and $G=(g_1,~g_2)$ for the crossed modules $\xymatrix{\mathfrak{C_1}=[N_1\ar[r]^{h_{\mathfrak{C_1}}} & M_1]}$ and $\xymatrix{\mathfrak{C_2}=[N_2\ar[r]^{h_{\mathfrak{C_2}}} & M_2]}$ is a crossed homomorphism $\gamma:~M_1\rightarrow N_2$ such that for all $a,~a'\in M_1$,
\begin{align}
\gamma(aa')=(^{f_1(a')}\gamma(a)).\gamma(a')
\end{align}

and the following conditions are satisfied.
\begin{enumerate}
\item $g_1(a)=f_1(a)h_{\mathfrak{C_2}}(\gamma(a^{-1}))$ for all $a\in M_1$ 
\item $g_2(b)=f_2(b)\gamma(h_{\mathfrak{C_1}}(b^{-1}))$ for all $b\in N_1$
\end{enumerate}
\end{defn}
\begin{remark}
A pointed natural transformation $PNT:~G\Rightarrow F$ between the crossed modules $\xymatrix{\mathfrak{C_1}=[N_1\ar[r]^{h_{\mathfrak{C_1}}} & M_1]}$ and $\xymatrix{\mathfrak{C_2}=[N_2\ar[r]^{h_{\mathfrak{C_2}}} & M_2]}$ that is a crossed homomorphism $\gamma:~M_1\rightarrow N_2$ is an isomorphism if there exists a pointed natural transformation $PNT':~F\Rightarrow G$ which is a crossed homomorphism $\gamma':~M_1\rightarrow N_2$ defined by $\gamma'(m)=\gamma^{-1}(m)$ for all $m\in M_1$.
\end{remark}
\begin{lem}
\label{600}
There exists a 2-category $\underline{\mathcal{XM}}$ whose objects are crossed modules, 1 arrows are strict morphisms between those crossed modules and 2 arrows are pointed natural transformations between those strict morphisms such that the pointed natural transformations $PNT:~G\Rightarrow G$ are the trivial pointed natural transformations where $G$ is a strict morphism between any crossed modules $\xymatrix{\mathfrak{C_1}=[N_1\ar[r]^{h_{\mathfrak{C_1}}} & M_1]}$ and $\xymatrix{\mathfrak{C_2}=[N_2\ar[r]^{h_{\mathfrak{C_2}}} & M_2]}$ in $\underline{\mathcal{XM}}$.
\end{lem}
\begin{proof}
We take $\mathfrak{X}=\underline{\mathcal{XM}}$. First, we need to show that $\underline{\mathcal{XM}}(\mathfrak{C_1},~\mathfrak{C_2})$ is a category whose objects are 1 arrows and morphisms are 2 arrows for the crossed modules $\xymatrix{\mathfrak{C_1}=[N_1\ar[r]^{h_{\mathfrak{C_1}}} & M_1]}$ and $\xymatrix{\mathfrak{C_2}=[N_2\ar[r]^{h_{\mathfrak{C_2}}} & M_2]}$ in $\underline{\mathcal{XM}}$.\\

The identity morphism is the trivial pointed natural transformation.\\

We define the composition of two pointed natural transformations $PNT_1:~G\Rightarrow F$ which is a crossed homomorphism $\gamma_1$ and $PNT_2:~F\Rightarrow E$ which is a crossed homomorphism $\gamma_2$ between the strict morphisms as $\gamma=\gamma_2. \gamma_1:~M_1\rightarrow N_2$. For all elements $a,~a'$ in $M_1$, we get\\

$\gamma(a.a')=\gamma_2(a.a') . \gamma_1(a.a')=(^{e_1(a')}\gamma_2(a))\gamma_2(a') .(^{f_1(a')}\gamma_1(a))\gamma_1(a')$\\

$=(^{e_1(a')}\gamma_2(a))\gamma_2(a')(^{e_1(a')h_{\mathfrak{C_2}}(\gamma_2((a')^{-1}))}\gamma_1(a))\gamma_1(a')$ since $f_1(a')=e_1(a')h_{\mathfrak{C_2}}(\gamma_2((a')^{-1}))$\\

$=(^{e_1(a')}\gamma_2(a))\gamma_2(a')(^{e_1(a')}\gamma_2((a')^{-1})\gamma_1(a)\gamma_2(a'))\gamma_1(a')$ since\\

$^{h_{\mathfrak{C_2}}(\gamma_2((a')^{-1}))}\gamma_1(a)=\gamma_2((a')^{-1}) . \gamma_1(a) . \gamma_2(a')$ by definition of crossed module\\

$=(^{e_1(a')}\gamma_2(a)). \gamma_2(a'). (^{e_1(a')}\gamma_2((a')^{-1})). (^{e_1(a')}\gamma_1(a)). (^{e_1(a')}\gamma_2(a')) . \gamma_1(a')$\\

$=(^{e_1(a')}\gamma_2(a)). \gamma_2(a').(^{e_1(a')}\gamma_2((a')^{-1})). \gamma_2(a') . \gamma_2((a')^{-1}) . (^{e_1(a')}\gamma_1(a)) . (^{e_1(a')}\gamma_2(a')) . \gamma_1(a')$\\

$=(^{e_1(a')}\gamma_2(a)). \gamma_2(a') . \gamma_2((a')^{-1} . a') . \gamma_2((a')^{-1}) . (^{e_1(a')}\gamma_1(a)) . (^{e_1(a')}\gamma_2(a')) . \gamma_2(a') . \gamma_2((a')^{-1}) . \gamma_1(a')$\\

$=(^{e_1(a')}\gamma_2(a)). (^{e_1(a')}\gamma_1(a)). \gamma_2(a' a'). \gamma_2((a')^{-1}). \gamma_1(a')$\\

$=(^{e_1(a')}\gamma_2(a). \gamma_1(a)). \gamma_2(a'). \gamma_1(a')$\\

$=(^{e_1(a')}\gamma(a)). \gamma(a')$.\\

$g_1(a)=f_1(a). h_{\mathfrak{C_2}}(\gamma_1(a^{-1}))$ since $PNT_1:~G\Rightarrow F$ is a pointed natural transformation which is a crossed homomorphism $\gamma_1$ and $f_1(a)=e_1(a). h_{\mathfrak{C_2}}(\gamma_2(a^{-1}))$ since $PNT_2:~F\Rightarrow E$ is a pointed natural transformation which is a crossed homomorphism $\gamma_2$. Hence, 
\begin{align*}
g_1(a)=e_1(a). h_{\mathfrak{C_2}}(\gamma_2(a^{-1})). h_{\mathfrak{C_2}}(\gamma_1(a^{-1}))=e_1(a). h_{\mathfrak{C_2}}((\gamma_2. \gamma_1)(a^{-1}))=e_1(a). h_{\mathfrak{C_2}}(\gamma(a^{-1})
\end{align*}

as desired. Similarly, we may show the other part. As a result, $\gamma$ is a crossed homomorphism and gives a pointed natural transformation $PNT_3:~G\Rightarrow E$.\\

It is clear that the composition is associative.\\

For all pointed natural transformations $PNT:~G\Rightarrow F$, the crossed homomorphism $\gamma$ and $id:~F\Rightarrow F$, the composition $id\star PNT:~G\Rightarrow F$ is equal to $PNT$. Similarly, we show the other part.\\

As a result, $\underline{\mathcal{XM}}(\mathfrak{C_1},~\mathfrak{C_2})$ is a category.\\

The mapping $\mathcal{F}_{\mathfrak{C_1} \mathfrak{C_2} \mathfrak{C_3}}:~\underline{\mathcal{XM}}(\mathfrak{C_2},~\mathfrak{C_3})\times \underline{\mathcal{XM}}(\mathfrak{C_1},~\mathfrak{C_2}) \rightarrow \underline{\mathcal{XM}}(\mathfrak{C_1},~\mathfrak{C_3})$ is a functor. We send each pair $(G,~F)$ to $G\circ F$ where $G=(g_1,~g_2)$, $F=(f_1,~f_2)$ and $G\circ F=(g_1\circ f_1,~g_2\circ f_2)$ are strict morphims between the corresponding crossed modules.\\

We send each pair of pointed natural transformations $PNT_2:~E\Rightarrow K$ which is a crossed homomorphism $\gamma_2$ and $PNT_1:~G\Rightarrow F$ which is a crossed homomorphism $\gamma_1$ to their composition $PNT_3:~E\circ G\Rightarrow K\circ F$. Here, $G=(g_1,~g_2)$, $F=(f_1,~f_2)$, $E=(e_1,~e_2)$ and $K=(k_1,~k_2)$ are strict morphisms. We need to define a crossed homomorphism $\gamma_3$.\\

We take $\gamma_3=(k_2\circ \gamma_1) . (\gamma_2 \circ g_1)$ and see it satisfies the required conditions to be a crossed homomorphism. We draw the following diagrams to see the relation between the strict morphisms.
\begin{align*}
\begin{tikzpicture}[out=145, in=145, relative]
\node (A) at (0,0) {$\mathfrak{C_1}$};
\node (B) at (3,0) {$\mathfrak{C_2}$};
\node (C) at (6, 0) {$\mathfrak{C_3}~~=$};
\draw[->, thick, double] (1.5,0.25) -- (1.5,-0.25) [xshift=5pt] node[right, midway] {$\gamma_1$} (B);
\draw[->, thick, double] (4.5,0.25) -- (4.5,-0.25) [xshift=5pt] node[right, midway] {$\gamma_2$} (C);
\path[->] (A) edge [bend left] node [above] {$G=(g_1,~g_2)$} (B);
\path[->] (A) edge [bend right] node [below] {$F=(f_1,~f_2)$} (B);
\path[->] (B) edge [bend left] node [above] {$E=(e_1,~e_2)$} (C);
\path[->] (B) edge [bend right] node [below] {$K=(k_1,k_2)$} (C);
\end{tikzpicture}
\begin{tikzpicture}[out=145, in=145, relative]
\node (A) at (0,0) {$\mathfrak{C_1}$};
\node (C) at (3,0) {$\mathfrak{C_3}$};
\draw[->, thick, double] (1.5,0.25) -- (1.5,-0.25) [xshift=5pt] node[right, midway] {$\gamma_3$} (C);
\path[->] (A) edge [bend left] node [above] {$E\circ G=(e_1\circ g_1,~e_2\circ g_2)$} (C);
\path[->] (A) edge [bend right] node [below] {$K\circ F=(k_1\circ f_1,~k_2\circ f_2)$} (C);
\end{tikzpicture}
\end{align*}

We also draw the following diagrams to understand the group homomorphisms better.
\begin{align*}
\begin{tikzpicture}
\node (A) at (0, 0) {$N_1$};
\node (B) at (0, -2) {$M_1$};
\node (C) at (2, -2) {$M_2$};
\node (D) at (2, 0) {$N_2$};
\node (E) at (4, 0) {$N_3$};
\node (F) at (4, -2) {$M_3$};
\path[->] (A) edge node [left] {$h_{\mathfrak{C_1}}$} (B);
\path[->] (A) edge node [above] {$g_2$} node [below] {$f_2$} (D);
\path[->] (B) edge node [above] {$g_1$} node [below] {$f_1$} (C);
\path[->] (D) edge node [right] {$h_{\mathfrak{C_2}}$} (C);
\path[->] (D) edge node [above] {$e_2$} node [below] {$k_2$} (E);
\path[->] (C) edge node [above] {$e_1$} node [below] {$k_1$} (F);
\path[->] (E) edge node [right] {$h_{\mathfrak{C_3}}~~\Rightarrow$} (F);
\path[->] (B) edge node [right] {$\gamma_1$} (D);
\path[->] (C) edge node [right] {$\gamma_2$} (E);
\end{tikzpicture}
\begin{tikzpicture}
\node (A) at (0, 0) {$N_1$};
\node (B) at (0, -2) {$M_1$};
\node (C) at (3, -2) {$M_3$};
\node (D) at (3, 0) {$N_3$};
\path[->] (A) edge node [left] {$h_{\mathfrak{C_1}}$} (B);
\path[->] (A) edge node [above] {$e_2\circ g_2$} node [below] {$k_2\circ f_2$} (D);
\path[->] (B) edge node [above] {$e_1\circ g_1$} node [below] {$k_1\circ f_1$} (C);
\path[->] (D) edge node [right] {$h_{\mathfrak{C_3}}$} (C);
\path[->] (B) edge node [right] {$\gamma_3$} (D);
\end{tikzpicture}
\end{align*}

For all $a,~a'\in M_1$, $b\in N_1$ and $c\in N_2$, $d,~d'\in M_2$, we get the following equalities by definition of $\gamma_1$ and $\gamma_2$.
\begin{enumerate}
\item \label{60} $\gamma_1(a.a')=(^{f_1(a')}\gamma_1(a)). \gamma_1(a')$
\item $g_1(a)=f_1(a). h_{\mathfrak{C_2}}(\gamma_1(a^{-1}))$
\item $g_2(b)=f_2(b). \gamma_1(h_{\mathfrak{C_1}}(b^{-1}))$
\item \label{61} $\gamma_2(d.d')=(^{k_1(d')}\gamma_2(d)). \gamma_2(d')$
\item $e_1(d)=k_1(d). h_{\mathfrak{C_3}}(\gamma_2(d^{-1}))$
\item \label{65} $e_2(c)=k_2(c). \gamma_2(h_{\mathfrak{C_2}}(c^{-1}))$
\end{enumerate}

For all $a'\in M_1$, we have 
\begin{align}
\label{63}
(k_1\circ g_1)(a')=k_1(f_1(a'). h_{\mathfrak{C_2}}(\gamma_1((a')^{-1})))=(k_1\circ f_1)(a'). (k_1\circ h_{\mathfrak{C_2}})(\gamma_1((a')^{-1})),
\end{align}
\begin{align}
(^{(k_1\circ f_1)(a')}(k_2\circ \gamma_1)((a')^{-1}))=k_2(^{f_1(a')}\gamma_1((a')^{-1}))\\
=k_2((^{f_1(a')}\gamma_1((a')^{-1})). \gamma_1(a'). \gamma_1((a')^{-1}))\\
\label{62} 
=k_2((a')^{-1}. a' . \gamma_1((a')^{-1}))=(k_2\circ \gamma_1)((a')^{-1}).
\end{align}

For all $a,~a'\in M_1$, we have\\

$\gamma_3(a . a')=(k_2\circ \gamma_1)(a.a') . (\gamma_2 \circ g_1)(a. a')=k_2((^{f_1(a')}\gamma_1(a)). \gamma_1(a')) . \gamma_2(g_1(a) . g_1(a'))$ by \ref{60}\\

$=k_2(^{f_1(a')}\gamma_1(a)). k_2(\gamma_1(a')). (^{(k_1\circ g_1)(a')}(\gamma_2 \circ g_1)(a)) . (\gamma_2 \circ g_1)(a')$ by \ref{61}\\

$=k_2(^{f_1(a')}\gamma_1(a)). k_2(\gamma_1(a')). (^{(k_1\circ f_1)(a'). k_1(h_{\mathfrak{C_2}}(\gamma_1((a')^{-1})))}(\gamma_2\circ g_1)(a)) . (\gamma_2 \circ g_1)(a')$ by \ref{63}\\

$=k_2(^{f_1(a')}\gamma_1(a)). k_2(\gamma_1(a')). (^{(k_1\circ f_1)(a'). (h_{\mathfrak{C_3}} (k_2\circ \gamma_1)((a')^{-1}))}(\gamma_2\circ g_1)(a)) . (\gamma_2 \circ g_1)(a')$ since $K$ is a strict morphism\\

$=k_2(^{f_1(a')}\gamma_1(a)). k_2(\gamma_1(a')). (^{(k_1\circ f_1)(a')}(k_2\circ \gamma_1)((a')^{-1}).(\gamma_2\circ g_1)(a). (k_2\circ \gamma_1)(a')). (\gamma_2 \circ g_1)(a')$ by definition of $\mathfrak{C_3}$\\

$=k_2(^{f_1(a')}\gamma_1(a)). k_2(\gamma_1(a')). (^{(k_1\circ f_1)(a')}(k_2\circ \gamma_1)((a')^{-1})) .(^{(k_1\circ f_1)(a')}(\gamma_2\circ g_1)(a)).\\
(^{(k_1\circ f_1)(a')}(k_2\circ \gamma_1)(a')). (\gamma_2 \circ g_1)(a')$\\

$=(^{(k_1\circ f_1)(a')}(k_2\circ \gamma_1)(a)).k_2(\gamma_1(a')). (^{(k_1\circ f_1)(a')}(k_2\circ \gamma_1)((a')^{-1})) .(^{(k_1\circ f_1)(a')}(\gamma_2\circ g_1)(a)).\\
(^{(k_1\circ f_1)(a')}(k_2\circ \gamma_1)(a')). (\gamma_2 \circ g_1)(a')$ since $K$ is a strict morphism\\

$=(^{(k_1\circ f_1)(a')}(k_2\circ \gamma_1)(a)). k_2(\gamma_1(a')). (k_2\circ \gamma_1)((a')^{-1}). (^{(k_1\circ f_1)(a')}(\gamma_2\circ g_1)(a)).\\
(^{(k_1\circ f_1)(a')}(k_2\circ \gamma_1)(a')). (\gamma_2 \circ g_1)(a')$ by \ref{62}\\

$=(^{(k_1\circ f_1)(a')}(k_2\circ \gamma_1)(a)). (^{(k_1\circ f_1)(a')}(\gamma_2\circ g_1)(a)). (^{(k_1\circ f_1)(a')}(k_2\circ \gamma_1)(a')). (\gamma_2 \circ g_1)(a')$\\

$=(^{(k_1\circ f_1)(a')}(k_2\circ \gamma_1)(a)). (^{(k_1\circ f_1)(a')}(\gamma_2\circ g_1)(a)). (k_2\circ \gamma_1)(a'). (\gamma_2 \circ g_1)(a')$ by \ref{62}\\

$=(^{(k_1\circ f_1)(a')}(k_2\circ \gamma_1)(a). (\gamma_2\circ g_1)(a)) . (k_2\circ \gamma_1)(a'). (\gamma_2 \circ g_1)(a')$\\

$=(^{(k_1\circ f_1)(a')}((k_2\circ \gamma_1). (\gamma_2\circ g_1))(a)). ((k_2\circ \gamma_1) . (\gamma_2 \circ g_1))(a')$\\

$=(^{(k_1\circ f_1)(a')}\gamma_3(a)). \gamma_3(a')$ as required.\\

The other conditions are satisfied.\\

For a crossed homomorphism $\xymatrix{\mathfrak{C}=[N\ar[r]^{h_{\mathfrak{C}}} & M]}$, the mapping $\mathcal{F}_{\mathfrak{C}}:~1\rightarrow \underline{\mathcal{XM}}(\mathfrak{C},~\mathfrak{C})$ is a functor taking the element $\star$ in $1$ to $id_h=(id,~id)$ and morphisms $\star \rightarrow \star$ to a trivial pointed natural transformation $PNT:~(id,~id)\Rightarrow (id,~id)$.
\begin{align*}
\begin{tikzpicture}
\node (A) at (0, 0) {$\underline{\mathcal{XM}}(\mathfrak{C_3},~\mathfrak{C_4}) \times \underline{\mathcal{XM}}(\mathfrak{C_2},~\mathfrak{C_3}) \times \underline{\mathcal{XM}}(\mathfrak{C_1},~\mathfrak{C_2})$};
\node (B) at (9, 0) {$\underline{\mathcal{XM}}(\mathfrak{C_2},~\mathfrak{C_4}) \times \underline{\mathcal{XM}}(\mathfrak{C_1},~\mathfrak{C_2})$};
\node (C) at (0,-2) {$\underline{\mathcal{XM}}(\mathfrak{C_3},~\mathfrak{C_4}) \times \underline{\mathcal{XM}}(\mathfrak{C_1},~\mathfrak{C_3})$};
\node (D) at (9, -2) {$\underline{\mathcal{XM}}(\mathfrak{C_1},~\mathfrak{C_4})$};
\draw[->, thick, double] (6,-0.75)--(3, -1.25) [xshift=5pt] node [right, midway] {$a_{\mathfrak{C_1} \mathfrak{C_2} \mathfrak{C_3} \mathfrak{C_4}}$} (B);
\path[->] (A) edge node [above] {$\mathcal{F}_{\mathfrak{C_2} \mathfrak{C_3} \mathfrak{C_4}} \times 1$} (B);
\path[->] (A) edge node [left, midway] {$1\times \mathcal{F}_{\mathfrak{C_1} \mathfrak{C_2} \mathfrak{C_3}}$} (C);
\path[->] (B) edge node [right] {$\mathcal{F}_{\mathfrak{C_1} \mathfrak{C_3} \mathfrak{C_4}}$} (D);
\path[->] (C) edge node [below] {$\mathcal{F}_{\mathfrak{C_1} \mathfrak{C_2} \mathfrak{C_4}}$} (D);
\end{tikzpicture}
\end{align*}

$a_{\mathfrak{C_1} \mathfrak{C_2} \mathfrak{C_3} \mathfrak{C_4}}(F,~G,~H):~(F\circ G)\circ H\Rightarrow F\circ (G\circ H)$ is a trivial pointed natural transformation and $(F\circ G)\circ H=F\circ (G\circ H)$ for all strict morphisms $F$, $G$ and $H$. 
$a_{\mathfrak{C_1} \mathfrak{C_2} \mathfrak{C_3} \mathfrak{C_4}}(F,~G,~H)$ is the identity morphism by assumption, hence it is an isomorphism and the required diagram is commutative. As a result, $a_{\mathfrak{C_1} \mathfrak{C_2} \mathfrak{C_3} \mathfrak{C_4}}$ is a natural isomorphism.\\

We show the other conditions in a similar way and see $\underline{\mathcal{XM}}$ is a 2-category.
\end{proof}

\end{document}